\documentclass[12pt]{article}
\usepackage{mlmodern}

\usepackage{pdfrender,xcolor}
\usepackage{graphicx,color}
\usepackage{amsmath,amsfonts,amssymb,amsthm,MnSymbol}
\usepackage{bbm}
\usepackage[colorlinks=true,linkcolor=blue,citecolor=blue,urlcolor=blue]{hyperref}
\usepackage{enumerate}
\usepackage{algorithm}
\usepackage{algpseudocode}
\usepackage{soul}
\usepackage{rsfso}
\usepackage{multirow}
\usepackage{booktabs}
\usepackage{multirow}
\usepackage{adjustbox}
\usepackage{caption}
\usepackage{subcaption}

\usepackage{natbib}
\newcommand{\E}{\mathbb{E}}
\newcommand{\R}{\mathbb{R}}
\newcommand{\Prob}{\mathbb{P}}
\newcommand{\Z}{\mathbb{Z}}

\DeclareMathOperator*{\argmin}{arg\,min}

\newcommand{\pend}{\hfill \thicklines \framebox(6.6,6.6)[l]{}}

\newenvironment{proof*}[1]{\noindent {\sc  #1} \rm}{\pend}

\newtheorem{theorem}{Theorem}[section]
\newtheorem{lemma}{Lemma}[section]
\newtheorem{assumption}{Assumption}[section]
\newtheorem{proposition}{Proposition}[section]

\newtheorem{remark}{Remark}[section]

\newtheorem{corollary}{Corollary}[section]
\newtheorem{definition}{Definition}[section]

\newcommand{\setsection}[2] {
	\setcounter{section}{#1}
	\setcounter{subsection}{0}
	\setcounter{equation}{0}
	\setcounter{conjecture}{0}
	\setcounter{assumption}{0}
	\setcounter{question}{0}
	\setcounter{definition}{0}
	\setcounter{theorem}{0}
	\setcounter{corollary}{0}
	\setcounter{lemma}{0}
	\setcounter{proposition}{0}
	\setcounter{remark}{0}
	\setcounter{appen}{0}
	\setsection*{\large \bf \thesection. #2}}

\usepackage{tikz}
\usetikzlibrary{positioning}
\usetikzlibrary{calc,fit}
\usepgflibrary{shapes.geometric}
\tikzset{
	server/.style={circle, inner sep=0.0mm, minimum width=.4cm,
		draw=green,fill=green!10,thick},
	ellipseserver/.style={ellipse, inner sep=0.0mm, minimum width=.8cm,
		draw=green,fill=green!10,thick},
	buffer/.style={rectangle, rounded corners=3pt,
		inner sep=0.0mm, minimum width=.7cm, minimum height=1cm,
		draw=orange,fill=blue!10,thick},
	vbuffer/.style={rectangle,rounded corners=3pt,
		inner sep=0.0mm,  minimum width=.7cm, minimum height=1cm,
		draw=orange,fill=purple!10,thick},
	wbuffer/.style={rectangle,rounded corners=3pt,
		inner sep=0.0mm,  minimum width=.7cm, minimum height=1cm,
		draw=orange,fill=orange!10,thick},
	routing/.style={circle,inner sep=0pt,minimum size=.45cm,
		draw=red,minimum width=.1cm, fill=red!30},
	state/.style={inner sep=1.0mm, rounded corners=2pt,
		draw=green,fill=green!10,thick},
	dot/.style={circle, inner sep=0.0mm, minimum width=.6cm,
		draw=blue,fill=blue!10,thick},
	task/.style={circle, inner sep=0.0mm, minimum width=.6cm,
		draw=blue,fill=blue!10,thick},
	data/.style={rectangle, inner sep=0.0mm, minimum width=.5cm,
		minimum height=.4cm,
		draw=red,fill=red!10,thick},
	center/.style={circle,inner sep=0.0mm, minimum width=0cm,
		draw=white,fill=white,thick },
}

\begin{document}
	\title{\bf \Large Explicit Steady-State Approximations for Parallel Server Systems with Heterogeneous Servers}
	
	\author{Yaosheng Xu \\ University of Illinois Urbana-Champaign} 
	\date{October 15, 2026}

	\maketitle 
	\begin{abstract}
		We study the steady-state performance of parallel-server systems under an immediate routing architecture with two sources of heterogeneity: server types and job classes, often subject to compatibility constraints. We focus on the weighted-workload-task-allocation (WWTA) routing policy, a load-balancing scheme known to be throughput-optimal for such systems. Under a relaxed complete-resource-pooling (CRP) condition, we prove a strong form of state-space collapse in heavy traffic and that the scaled server workloads jointly converge in distribution to a deterministic vector multiplied by an exponential random variable, whose mean is explicitly given by system primitives. Our analysis yields three main insights. First, the conventional heavy-traffic assumption of a unique static allocation plan can be dropped; a relaxed CRP condition suffices. Second, the limiting workload distribution is shown to be independent of the local scheduling policies at the servers, allowing substantial flexibility. Third, the inefficient (non-basic) activities identified by the static allocation problem are proved to receive an asymptotically negligible fraction of routed jobs and service effort, even though WWTA has no prior knowledge of which activities are non-basic, highlighting its robustness to changing arrival rates.
		
	\end{abstract}
	\section{Introduction}
Many engineering systems involve the allocation of multiple classes of jobs or tasks among multiple types of service resources. Examples include call centers, manufacturing systems, data centers, wireless networks, and distributed computing platforms. Despite their different application contexts, these systems share common operational challenges arising from heterogeneity in job classes and service resources, differences in processing rates, and compatibility constraints on which resources can serve which jobs. The allocation and scheduling of jobs across resources are therefore central to system performance.

Parallel server systems provide a natural queueing framework for modeling the job allocation and scheduling problems arising in these applications. In this paper, we study general parallel server systems under the output-queued (OQ) architecture \citep{Stolyar2005}, as illustrated in Figure~\ref{fig:g}, with an arbitrary but fixed number of job classes and server types. In an OQ system, each arriving job is immediately assigned to one of the eligible servers and then waits in a queue associated with that server for processing. Once assigned, it cannot be served by other servers. We consider systems with general compatibility constraints and with service rates that may depend on both the job class and the server type.

\begin{figure}
		\centering
		\begin{tikzpicture}[inner sep=1.5mm]
			\node[server,minimum size=0.45in] at (6,4) (S1)  {$S_1$};
			\node[buffer] (B1) [above=0.25in of S1.north] {$B_1$};
			\node[server,minimum size=0.45in] (S2) [right=0.7in of S1] {$S_2$};
			\node[server,minimum size=0.45in] (S3) [right=0.7in of S2] {$S_k$};
			\node[server,minimum size=0.45in] (S4) [right=0.7in of S3] {$S_K$};
			\node[buffer] (B2) [above=0.3in of S2.north west] {$B_2$};
			\node[buffer] (B3) [above=0.3in of S2.north east] {$B_3$};
			\node[buffer] (B4) [above=0.25in of S3.north] {$B_j$};
			\node[buffer] (B5) [above=0.3in of S4.north west] {$B_{J-1}$};
			\node[buffer] (B6) [above=0.3in of S4.north east] {$B_{J}$};
			\node[routing] (R1) [above=0.4in of B1.north] {$R_1$};
			\node[routing] (R2) [above=0.4in of B4.north] {$R_i$};
			\node[routing] (R3) [above=0.4in of B6.north] {$R_I$};
			\node[center](C)[right=0.3in of B3]{...};
			\node[center](C)[right=0.3in of B4]{...};
			\node[center](C)[right=1in of R1]{...};
			\node[center](C)[right=0.5in of R2]{...};	
			\coordinate (source1) at ($ (B1.north) + (0in, 0.9in)$) {};
			\coordinate (source2) at ($ (B4.north) + (0in, 0.9in)$) {};
			\coordinate (source3) at ($ (B6.north) + (0in, 0.9in)$) {};
			\coordinate (departure1) at ($ (S1.south) + (0in, -.25in)$) {};
			\coordinate (departure2) at ($ (S2.south) + (0in, -.25in)$) {};
			\coordinate (departure3) at ($ (S3.south) + (0in, -.25in)$) {};	
			\coordinate (departure4) at ($ (S4.south) + (0in, -.25in)$) {};	
			\draw[thick,->] (source1) -| (R1);
			\draw[thick,->] (source2) -| (R2);
			\draw[thick,->] (source3) -| (R3);	
			\draw[thick,->] (R1.south) -| (B1);
			\draw[thick,->] (R1.south) -- ++(B2);
			\draw[thick,->] (R2.south) -- ++ (B3);	
			\draw[thick,->] (R2.south) -| (B4);
			\draw[thick,->] (R3.south) -- ++(B6);
			\draw[thick,->] (R2.south) -- ++ (B5);	
			\draw[thick,->] (S1.south) -| (departure1);
			\draw[thick,->] (S2.south) -| (departure2);
			\draw[thick,->] (S3.south) -| (departure3);			\draw[thick,->] (S4.south) -| (departure4);	
			\draw[thick,->] (B1.south) -| (S1);
			\draw[thick,->] (B2.south) -| (S2.north west);
			\draw[thick,->] (B3.south) -| (S2.north east);
			\draw[thick,->] (B4.south) -| (S3);
			\draw[thick,->] (B5.south) -| (S4.north west);
			\draw[thick,->] (B6.south) -| (S4.north east);
			\path[thick]
			;
			
			\node [below ,align=left] at ($(departure1)+(0in,0)$) {Server 1\\ departures};
			\node [below,align=left] at ($(departure2)+(0in,0)$) {Server 2\\departures};
			\node [below,align=left] at ($(departure3)+(0in,0)$) {Server k\\departures};
			\node [below,align=left] at ($(departure4)+(0in,0)$) {Server K\\departures};
			\node [above,align=left] at (source1.south) {class 1\\ arrivals};
			\node [above,align=left] at (source2.south) {class i\\ arrivals};
			\node [above,align=left] at (source3.south) {class I\\ arrivals};
			
			\node [left,align=left] at ($(source1)+(0in,-.2in)$) {\textbf{$\lambda_1$}};
			\node [right,align=left] at ($(source2)+(0in,-.2in)$) {\textbf{$\lambda_i$}};
			\node [right,align=left] at ($(source3)+(0in,-.2in)$) {\textbf{$\lambda_I$}};
			
			\node [left,align=left] at ($(B1)+(0in,-.4in)$) {\textbf{$\mu_1$}};
			\node [left,align=left] at ($(B2)+(0in,-.4in)$) {\textbf{$\mu_2$}};
			\node [right,align=left] at ($(B3)+(0in,-.4in)$) {\textbf{$\mu_3$}};	
			\node [right,align=left] at ($(B4)+(0in,-.4in)$) {\textbf{$\mu_j$}};	
			\node [left,align=left] at ($(B5)+(0in,-.4in)$) {\textbf{$\mu_{J-1}$}};	
			\node [right,align=left] at ($(B6)+(0in,-.4in)$) {\textbf{$\mu_J$}};	
		\end{tikzpicture}
		\caption{General output-queued parallel server system} 
		\label{fig:g}
	\end{figure}
    
Under this general OQ parallel-server setting, we focus on the steady-state performance of the weighted-workload-task-allocation (WWTA) routing policy, which will be formally introduced in Section~\ref{sec:pss}. Roughly speaking, WWTA compares the workloads of compatible servers and favors servers with smaller workloads while accounting for differences in service rates. It can be regarded as a generalization of the join-the-shortest-queue (JSQ) policy, which simply routes each arriving job to a server with the shortest queue. After job allocation, a server may process multiple job classes. Since service rates can differ across the job classes, routing alone does not specify which job class the server should process next. A local scheduling policy is therefore needed at each server to make this decision. Throughout the paper, we allow each server to implement \emph{any} non-idling scheduling policy.

There are several reasons to study the WWTA routing policy. First, WWTA is robust to varying arrival rates: it dynamically routes jobs to servers without requiring arrival-rate information for each job class. The JSQ policy is another popular scheme that likewise does not use the arrival rates. However, this widely adopted routing policy can become unstable when service rates depend on different job classes; see Section~\ref{sec:xmodel} for an example. Second, WWTA is throughput-optimal in general OQ parallel server systems with heterogeneous job classes and server types; see \cite{DaiHarr2020} for criteria on subcriticality and throughput optimality. Third, WWTA is a special case of the MinDrift routing rule \citep{Stolyar2005} with a quadratic cost function and therefore inherits its asymptotic server-workload optimality. Taken together, these properties make WWTA a natural candidate for load balancing when service rates vary across both job classes and server types.
    
While the stability and asymptotic optimality of WWTA have been studied, its steady-state performance remains less understood. In this paper, we characterize its steady-state behavior under the heavy traffic regime. Beyond the study of WWTA itself, steady-state analysis remains largely unexplored for general many-to-many parallel-server systems with multiple job classes processed by multiple server types, particularly in the presence of compatibility constraints. These features introduce structural and analytical challenges that are usually absent from simpler single-class, one-to-many, or many-to-one systems. Despite these challenges, we obtain explicit characterizations of steady-state performance in heavy traffic.

In the study of parallel-server systems, heavy-traffic and resource-pooling conditions are commonly imposed and characterized through the static allocation problem. As formulated in \cite{HarrLope1999}, the static allocation problem is a linear program in terms of the arrival and service rates. It aims to balance server utilization and determine nominal service allocations across job classes and servers, where the processing of a particular job class by a particular server is referred to as an \emph{activity}. Classical formulations often assume that the optimal solution to the static allocation problem is unique. Under this assumption, an activity is called \emph{basic} if it receives positive allocation in the optimal allocation plan and \emph{non-basic} if it receives zero allocation. This unique allocation plan has been utilized in policy design and performance analysis. In this paper, we allow the static allocation problem to have multiple optimal solutions. An activity may therefore receive different allocations across optimal solutions, so the classical definitions and treatments no longer apply directly. We extend the notions of heavy traffic, resource pooling, and basic and non-basic activities to this more general setting.

It is worth emphasizing that our heavy-traffic and resource-pooling conditions are needed only for performance analysis and not for implementing WWTA. The policy itself does not use the arrival rates, the static allocation problem or its solutions, or the classification of activities as basic or non-basic.

Our heavy-traffic analysis builds on the Basic Adjoint Relationship (BAR) approach \citep{BravDaiMiya2017}. The BAR provides a steady-state identity involving the system generator, and different performance quantities can be studied through appropriate choices of test functions. When applied to exponential test functions, the BAR yields equations involving Laplace transforms of the stationary system state. We use these transform equations, together with state-space collapse and transform-level flow-balance arguments, to identify the limiting Laplace transform, which in turn characterizes the limiting distribution of the steady-state workload in heavy traffic.

We summarize our main contributions and new analytical ingredients below.

\subsection{Our Contribution}
\begin{enumerate}
\item \textbf{State-space collapse and steady-state heavy-traffic limit.}
 We prove in Theorem~\ref{thm:main} that the scaled steady-state workload vector converges in distribution to a deterministic vector multiplied by an exponential random variable, whose mean is given explicitly in terms of the system primitives. The resulting limiting distribution is independent of the choice of non-idling service scheduling policy. A key ingredient in establishing this limit is the state-space-collapse result in Proposition~\ref{prop:sscmoment}. Our state-space-collapse result provides moment bounds of all orders, going beyond the second-order bound needed to establish the limiting distribution. Specifically, the moments of the difference between the largest and smallest weighted server workloads remain bounded for all sufficiently small values of the heavy-traffic parameter. The proof of this result uses a family of test functions and a new recursive argument that controls each higher-order moment in terms of lower-order moments, allowing the bounds to be established through mathematical induction. Unlike the drift-based approach developed in \cite{EryiSrik2012} and subsequent work, which relies on bounded-support assumptions on the stochastic primitives, our recursive approach has the potential to require only sufficiently high moment assumptions.
 
\item \textbf{Relaxed CRP and non-unique static allocation.}
We establish the steady-state heavy-traffic limit without requiring the static allocation problem to have a unique optimal solution, an assumption commonly imposed in the literature. To this end, we formulate relaxed heavy-traffic and relaxed CRP conditions that allow multiple optimal solutions to the static allocation problem. Under these conditions, we show in Lemma~\ref{lem:pd2} that the corresponding dual problem has a unique optimal solution. We establish dual uniqueness directly through a server communication argument, providing an alternative to the treatments in \cite{Stolyar2005,AtaCastRei2024}. We use the components of this unique dual solution to quantify the heterogeneous class-level and server-level workload contributions and characterize the steady-state behavior under non-unique static allocation.

\item \textbf{Adaptivity of WWTA to non-basic activities.}
Under the relaxed CRP condition, each optimal solution to the static allocation problem specifies its own basic and non-basic activities. We extend the notion of a non-basic activity to our setting by calling an activity \emph{non-basic} if it receives zero allocation in every optimal solution. Such activities represent inefficient class-server assignments that should not be used in any optimal allocation plan. Identifying non-basic activities generally requires knowledge of the arrival rates and the corresponding static allocation problem. WWTA requires neither. While the vanishing of diffusion-scaled queues associated with non-basic activities was identified as a necessary condition for workload minimization in \cite{Stolyar2005}, we establish quantitative steady-state results on how WWTA adaptively suppresses non-basic activities. In particular, we show in Lemmas~\ref{lem:idle} and~\ref{lem:nbcon} that (a) WWTA routes only a negligible fraction of jobs to non-basic activities, (b) each server spends only a negligible fraction of time on non-basic activities, and (c) the steady-state number of jobs associated with non-basic activities is negligible under the heavy-traffic scaling.  

\item \textbf{Asymptotic transform-level flow balance.}
In closing the transform equations derived from the BAR to establish the heavy-traffic limit, the two sources of heterogeneity across job classes and server types pose a key challenge. Specifically, the transform BAR contains activity-specific exponentially weighted service quantities that cannot be directly expressed in terms of the transform itself. We resolve this difficulty by establishing in Lemma~\ref{lem:transflowbalance} an asymptotic transform-level flow-balance identity for each job class. For a Laplace-transform parameter $\theta\leq0$, this identity links the class-level arrival flow with the corresponding activity-specific service flows under the same exponential weighting used in the BAR. At $\theta=0$, the identity reduces to the heavy-traffic limit of the usual steady-state flow balance. For $\theta<0$, we prove it by perturbing the exponential test function along a single job class, first establishing the resulting identity for a positive perturbation parameter and then recovering the desired identity by interchanging the heavy-traffic and perturbation-parameter limits. To the best of our knowledge, establishing transform-level flow balance separately for each job class under the full-state exponential weighting used in the BAR is new.
\end{enumerate}

Our theoretical results hold under any non-idling scheduling policy. We further complement these results with numerical case studies that explore the roles of routing and scheduling under the OQ and corresponding input-queued (IQ) architectures. The key message is that, under WWTA, load balancing is handled at the routing stage, leaving flexibility in the choice of local service scheduling policy. In Section~\ref{sec:discussion}, an X model illustrates how routing and scheduling interact in determining stability. A W model then shows that different non-idling scheduling policies under WWTA can lead to differences in delay performance, motivating the optimization of scheduling policies while retaining the performance guarantees established in this paper.

This paper remains within the continuous-time Markov chain framework by assuming Poisson arrivals and exponentially distributed service times. A possible route toward general renewal arrivals and service-time distributions is to augment the state with residual interarrival and service times and replace the present Markovian BAR with its Palm version. This, however, is only the starting point: the main analytical arguments developed here would need suitable counterparts in the general-distribution setting, bringing additional technical challenges beyond those addressed in this paper. Developing these results under general distributions remains a future research direction.

We conclude this section with a review of the related literature. The rest of the paper is organized as follows. Section~\ref{sec:pss} describes the general OQ parallel-server system and formally defines the WWTA routing policy and the class of non-idling scheduling policies considered in the paper. Section~\ref{sec:main} develops the heavy-traffic framework under the relaxed CRP condition and presents the main result on the limiting steady-state workload distribution. Sections~\ref{sec:outline1}, \ref{sec:ssc_moment}, and \ref{sec:outline2} establish the supporting results, with the state-space-collapse result in Section~\ref{sec:ssc_moment} serving as a key ingredient. These results are then combined in Section~\ref{sec:mainproof} to prove the main theorem. Section~\ref{sec:discussion} presents numerical studies that further examine routing and scheduling under the OQ and corresponding IQ architectures.

\subsection{Related Literature}

Routing and scheduling in general parallel-server systems have been studied extensively in the literature under both the OQ and IQ architectures. In OQ systems, jobs are assigned to servers upon arrival and subsequently scheduled locally at each server, whereas in IQ systems, jobs remain in class-specific queues and server assignment is determined through scheduling decisions. \cite{HarrLope1999} developed the static allocation framework for IQ systems and characterized the CRP condition under which multiple servers effectively behave as a single pooled resource in heavy traffic. Although originally developed in a different architecture, this framework and the associated heavy-traffic and CRP conditions also underlie our steady-state analysis of OQ systems.

A substantial line of work on IQ systems building on the static allocation framework has focused on the design of scheduling policies and their asymptotic optimality in heavy traffic. Early results for a two-class two-server system were established by \cite{Harr1998}, followed by extensions to more general many-to-many parallel-server systems. \cite{MandStol2004} established asymptotic optimality of the generalized $c\mu$ rule. In a discrete-time generalized switch model, \cite{Stol2004} studied MaxWeight scheduling (the generalized $c\mu$ rule with quadratic costs) and established state-space collapse and heavy-traffic optimality. \cite{BellWill2005} considered a continuous-review threshold policy that exploits the tree structure formed by the basic activities and characterized its heavy-traffic performance. \cite{AtaKumar2005} extended this line of research to more general stochastic processing networks through the analysis of a discrete-review policy.

For OQ systems, relatively few studies have considered policy design and analysis for general many-to-many systems. A notable work is \cite{Stolyar2005}, which introduced the MinDrift routing rule for such systems. Under the CRP condition and arbitrary work-conserving scheduling rules at the servers, MinDrift was proved to asymptotically minimize the server-workload process among OQ policies; the resulting minimal server-workload process matches the minimal customer-workload process attainable in the corresponding IQ systems. General OQ systems with structured service models have also been studied in cloud-computing and data-locality settings; see, e.g., \cite{MaguSrikYing2014,WangZhuYing2016,XieYekkLu2016}. Other policies studied for OQ systems have mainly focused on a single class and multiple servers. Examples include the widely adopted JSQ and power-of-$d$ routing policies; \cite{TehWard2002} studied a threshold routing policy for a two-server system; \cite{ChenYe2012} developed a balanced routing policy based on service-rate-weighted sampling.

    Much of the aforementioned literature built on the static allocation problem relies on its unique optimal solution. Studies that consider nonuniqueness include \cite{Stolyar2005} and \cite{AtaCastRei2024}. \cite{Stolyar2005} introduced a first-order complete-resource-pooling (FO-CRP) condition formulated geometrically in terms of the arrival-rate vector. \cite{AtaCastRei2024} assumed uniqueness of the dual solution and provided an explicit characterization under a decomposable service-rate structure. The connection between their conditions and our relaxed heavy-traffic and CRP conditions is discussed in Section~\ref{sec:main}. In a different direction, \cite{PesicWill2016,HurtMag2022,CardBorsVanL2024} studied heavy-traffic behavior in settings without the CRP condition, allowing for multidimensional limiting behavior.

The WWTA policy is closely related to the weighted-workload routing policy proposed by \cite{XieYekkLu2016} in a three-level data-locality setting, which can be viewed as a predecessor of WWTA. Coupled with a priority scheduling rule, their policy was shown to be throughput- and heavy-traffic optimal. \cite{DaiHarr2020} subsequently studied WWTA in a general OQ parallel-server system and established its maximal stability under arbitrary non-idling scheduling policies. From a different perspective, \cite{Stolyar2005} studied asymptotic workload optimality in OQ parallel-server systems under heavy traffic and introduced the MinDrift routing rule. The WWTA policy studied in this paper is a special case of MinDrift with a quadratic cost function; we discuss this connection in detail in Section~\ref{sec:pss}.

The BAR approach we build upon provides a direct method for steady-state heavy-traffic analysis. Developed for continuous-time queueing networks, the BAR approach can accommodate general interarrival and service-time distributions; see, for example, \cite{BravDaiMiya2017,BravDaiMiya2024,DaiGlynXu2026}. In related work, \cite{GuangXuDai2025} adapted the BAR approach and the recursive moment-bound argument developed in this paper to the simpler setting of single-class OQ systems under JSQ and power-of-two routing policies, and established their steady-state heavy-traffic limits with general distributions. Other direct methods for steady-state analysis include the drift-based method and the transform method. The drift-based method uses polynomial test functions and steady-state drift equations to obtain moment bounds and heavy-traffic limits; see, for example, \cite{EryiSrik2012,MaguSrik2016}. For state-space collapse, \cite{EryiSrik2012} established exponential state-space collapse under bounded-support assumptions on the stochastic primitives, an approach subsequently adopted in this line of work. The transform method based on exponential test functions characterizes limiting stationary distributions through their transforms; see \cite{HurtMagu2020} for JSQ load-balancing systems and generalized switch models, whose state-space-collapse analysis builds on the drift-based approach of \cite{EryiSrik2012}. Stein's method provides another direct approach and can further establish rates of convergence to the limiting distribution; see \cite{Gurv2014,BravDai2017}. A key challenge in applying this approach is to identify an appropriate family of test functions. A recent tutorial \cite{VarmaJhunHurtMagu2026} provides a detailed discussion and comparison of these direct methods for steady-state heavy-traffic analysis.

In contrast to these direct methods, classical heavy-traffic analysis often proceeds through process-level diffusion approximations, with an interchange-of-limits argument used to establish convergence of stationary distributions; see, for example, \cite{GamaZeev2006,BudhLee2009,Gurv2014a,YeYao2018}. Recently, \cite{Ye2023} used this approach to study stationary heavy-traffic performance of several routing policies based on different levels of state information.

\cite{BoonBoxKellMiya2017} developed distributional balance relations, including transform-level relations, for multiclass multiserver queues and general queueing networks. These relations connect queue-length distributions or transforms observed at different arrival and departure epochs, in settings that may allow state-dependent routing. When applied to our setting, their transform-level relations can provide an aggregate balance across classes, whereas our identity establishes an asymptotic flow balance separately for each job class. This class-level balance allows us to handle the activity-specific terms arising in the limiting BAR.

Steady-state heavy-traffic performance in parallel-server systems has also been investigated under different service disciplines and mechanisms. Under the FCFS-ALIS discipline, \cite{AfecCaldGupt2021,HillCaldGupt2024} characterized waiting times based on the compatibility structure and a decomposition into CRP components. Redundancy systems, in which jobs may be replicated across multiple servers, were studied in \cite{AyesBodaVerl2018,CardBorsVanL2022,Stol2022}, including heavy-traffic and mean-field asymptotic regimes.

A separate line of research considers large-scale asymptotic regimes in which the number of servers grows. Many-server heavy-traffic results include \cite{Atar2005,StolTezc2010,DaiTezc2011}. Under subcritical large-scale scaling, related results have been obtained in \cite{Stol2015,WengZhouSrik2020,RutMuk2022,ZhaoMuk2026}. These regimes differ from the fixed-dimensional heavy-traffic setting considered in this paper.

	\section{Parallel Server Systems and Policies}\label{sec:pss}
	
	Consider a general OQ parallel-server system with heterogeneous job classes and server types, as illustrated in Figure~\ref{fig:g}. Assume that there are $K$ servers processing $I$ classes of jobs. Let $\mathcal{I} = \{1,\ldots, I\}$ be the set of classes, and $\mathcal{K} = \{1,\ldots, K\}$ be the set of servers.  For each class $i\in\mathcal{I}$, the job-arrival process is assumed to be Poisson with rate $\lambda_i>0$.  Each server $k\in \mathcal{K}$ may be cross-trained and can process jobs from multiple classes. The processing times of class $i$ jobs at server $k$ are independent and identically distributed exponential random variables with mean $m_{ik}>0$. We denote by $\mathcal{I}(k)\subset\mathcal{I}$ the
set of classes that server $k$ can process, and by
$\mathcal{K}(i)\subset\mathcal{K}$ the set of servers that can process class $i$ jobs. The total number of activities is denoted by $J$, and we index the activities by $\mathcal{J}=\{1,\ldots,J\}$, where $J\leq I\times K$. Each activity $j\in\mathcal{J}$ is uniquely associated with a class-server pair $(i,k)$, and we write $j=(i,k)$. We define $\mu_j=\mu_{ik}=1/m_{ik}$ as the service rate of activity $j$. 
	
	Upon arrival, each class $i$ job is immediately routed to one of the
	servers in $\mathcal{K}(i)$, say, server
	$k\in \mathcal{K}(i)$, according to a routing algorithm. 
	If server $k$ is idle, the job begins service
	immediately; otherwise, it waits in the corresponding buffer until the server is ready to process it according to a scheduling policy. Each server $k$ maintains multiple
	buffers, one for each activity type $j=(i,k)$ with
	$i\in \mathcal{I}(k)$. In this paper, we focus on the
	weighted-workload-task-allocation (WWTA) load-balancing
	algorithm, also known as the WWTA routing policy, which we formally introduce next. 
	
	\paragraph*{Weighted-workload-task-allocation (WWTA) routing policy.}
    The WWTA policy makes routing decisions by comparing server workloads. Since the actual remaining processing requirements of individual jobs are not observable in many applications, we define the workload using the observable queue lengths and mean service times. Specifically, for server $k$, let 
	$$W_k({z})=\sum_{i\in \mathcal{I}(k)}m_{ik}z_{ik},$$
	where $z=(z_{ik})$ is the vector of queue lengths, and $z_{ik}$ is the number of jobs in buffer $j=(i,k)$,
	including possibly the one in service. The WWTA policy routes an arriving class-$i$ job to a server
   \begin{equation}\label{eq:wwta} k\in\argmin_{\ell\in \mathcal{K}(i)}m_{i\ell}W_{\ell}({z}).\end{equation}
  Ties are broken arbitrarily.

\begin{remark}
The two factors in $m_{ik}W_k(z)$ have different interpretations: $m_{ik}$ represents the additional workload introduced by the arriving job if routed to server $k$, while $W_k(z)$ reflects the existing congestion at server $k$. Their product captures the first-order impact of assigning the job to that server on a quadratic workload cost. This interpretation is closely related to the MinDrift(Q) rule of \cite{Stolyar2005}, which, under our notation, routes an arriving class-$i$ job according to
$$k\in\argmin_{\ell\in\mathcal{K}(i)}
m_{i\ell} h_{\ell}'\bigl(W_\ell(z)\bigr),$$
where $h_k$ is a cost function associated with the workload of server $k$. With the quadratic choice $h_k(w)=w^2/2$, we have $h_k'(w)=w$, and the MinDrift(Q) rule reduces exactly to WWTA. Thus, WWTA greedily minimizes the first-order increase in the aggregate quadratic workload cost associated with routing each new arrival.
\end{remark}

\paragraph*{Service policy.}
Since jobs in different buffers may have different service rates, we further specify the scheduling policy that dictates, for each server, which buffer to process next. It is known that under the WWTA routing policy, any non-idling scheduling policy is throughput optimal \citep{DaiHarr2020}. By non-idling, we mean that each server is busy processing jobs whenever there are jobs waiting in its buffers. Throughout the paper, we use $\mathcal{P}=(P_{ik}(z))$ to represent a non-idling scheduling policy, where $P_{ik}(z)$ denotes the fraction of effort that server $k$ allocates
to class $i$ jobs when the queue length vector is $z$, with $P_{ik}(z)=0$ when $z_{ik}=0$. In particular, non-idling means
\begin{equation}\label{eq:nonidle}
\sum_{i\in\mathcal{I}(k)}P_{ik}(z)=1
\quad \text{if } \sum_{i\in\mathcal{I}(k)}z_{ik}>0, \qquad \text{for each } k\in\mathcal{K}.
\end{equation}
Again, our main results are independent of the specific scheduling policy. Explicitly introducing $P_{ik}(z)$ allows us to rigorously establish the results based on the underlying queue-length dynamics. Nevertheless, equation~\eqref{eq:nonidle} is the only property of scheduling policies required for our analysis. Several candidate scheduling policies, together with their explicit expressions for $P_{ik}(z)$, are discussed in Section~\ref{sec:service_policy}.
	
	\section{Assumptions and Main Results}\label{sec:main}
	In this section, we formally present the main results. We start by introducing two critical assumptions: the (relaxed) heavy-traffic condition and the (relaxed) complete-resource-pooling condition. Both assumptions are formulated through solutions to a linear
	program (LP) developed in
	\cite{HarrLope1999}. We adopt the notation used therein, which centers on the concept of activities. In particular, an $I\times J$ \emph{constituency} matrix $C$ and
	a $K\times J$ \emph{resource-consumption} matrix $A$ are defined as follows:
	$$C_{i j}=\left\{\begin{array}{ll}1, & \text { if activity } j \text { processes class } i; \\ 0, & \text { otherwise. }\end{array}\right.$$
	
	$$A_{k j}=\left\{\begin{array}{ll}1, & \text { if server } k \text { performs activity } j; \\ 0, & \text { otherwise. }\end{array}\right.$$
	Each activity $j\in\mathcal{J}$ is uniquely associated with a class $i$ and a server $k$. Thus, the matrices $C$ and $A$ have exactly one nonzero element in each column. In contrast, a row of $C$ or $A$ may contain multiple nonzero elements, since a class or server may be involved in different activities. We denote the service rate vector by
	$$\mu=(\mu_{1},\ldots, \mu_J)^\top,$$
	where $\mu_j=\mu_{ik}$ for activity $j=(i,k)$ and $\top$ denotes the transpose.
	Define the \emph{output} matrix
	$$ R = C \operatorname{diag}(\mu),
	$$
	where, for activity $j=(i,k)$, $R_{ij}$ is the job departure rate from the class-$i$ buffer when server $k$ devotes all its effort to this activity.
	
	We introduce the
	following LP, which is known as the 
	\emph{static allocation problem}:
	\begin{equation}
		\label{eq:primal}
		\begin{aligned}
			\min\quad &\rho \\
			\text{s.t.} \quad&R x=\lambda,\\
			\quad&
			Ax \leq \rho 
			e,\\
			\quad& x, \rho\geq 0,
		\end{aligned}
	\end{equation}
	where $\lambda = (\lambda_1,\ldots,\lambda_I)^\top$, $x=(x_1,\ldots,x_J)^\top$ and $e=(1,\ldots,1)^\top\in \mathbb{R}^K$. The vector $x$ can be regarded as a processing plan, with each element $x_j$ interpreted as the long-run proportion of time that activity $j$ is processed by its server, and $\rho$ interpreted as the long-run utilization of the busiest server. 
    
	We introduce our first assumption as follows:
	\begin{assumption}[(Relaxed)  Heavy Traffic]\label{assu1}
		The parallel server system is assumed to be in (relaxed)
		\emph{heavy traffic}; namely,  the static allocation problem
		$(\ref{eq:primal})$ has an optimal solution $(x^*, \rho^*)$ satisfying         \begin{align}\label{eq:ht}
			\rho^*=1  \quad \text{ and } \quad  A x^{*}=e.                \end{align}
    The optimal processing plan $x^*$ is not required to be unique.
	\end{assumption}
	
	\begin{remark}
	Uniqueness of the optimal processing plan is assumed in much of the related literature; see, for example, \cite{HarrLope1999,Harr2000,BellWill2005}. In contrast, Assumption~\ref{assu1} allows multiple optimal processing plans, each of which keeps all servers fully utilized. In \cite{AtaCastRei2024}, the authors likewise relaxed the uniqueness of the optimal processing plan, referring to the resulting condition as the ``extended heavy traffic condition"; the extreme points of the optimal solution set serve as control modes in their diffusion control problem.
	\end{remark}
	The notions of basic and non-basic activities, when the optimal processing plan $x^*$  satisfying (\ref{eq:ht}) need not be unique, are defined as follows.  Activity $j={(i,k)}$ is called a \textit{basic activity associated with an optimal processing plan $x^*$} if $x^*_{j}> 0$; otherwise, it is called a \textit{non-basic activity associated with $x^*$}. We call an activity \textit{non-basic} if it is non-basic for every optimal processing plan $x^*$.  We define communication between servers as follows: 
	\begin{definition}
		Servers $k$ and $k'$ are said
		to \textit{communicate directly} if there exists an optimal processing plan $x^*$ such that both $j =(i,k)$ and $j'=(i,k')$, for some class $i$, are basic activities associated with $x^*$. Servers $k$ and $k'$ are said
		to \textit{communicate}, 
		if there exists a sequence of servers $k_1,\ldots, k_{\omega}$ such that $k_1 = k, k_{\omega}=k'$, and $k_{\alpha}$ communicates directly with $k_{\alpha+1}$  for $\alpha=1,\ldots,\omega-1$.  
	\end{definition}
	\begin{assumption}[(Relaxed) Complete Resource Pooling]\label{assu2}
	Under Assumption~\ref{assu1}, all servers communicate, where different optimal processing plans may be used to establish direct communication between different pairs of servers.
	\end{assumption}
	\begin{remark}
		A related formulation is given by \cite{Stolyar2005}, which defines the FO-CRP condition geometrically in terms of $\lambda$. Although formulated from different perspectives, Assumptions~\ref{assu1} and~\ref{assu2} together are equivalent to the FO-CRP condition; see Lemma~2 of \cite{Stolyar2005}. Despite this equivalence, we retain our relaxed-CRP formulation and establish dual uniqueness directly through a server communication argument, providing an alternative to the geometric argument in \cite{Stolyar2005}.
	\end{remark}
    
	 The dual LP to the static allocation problem (\ref{eq:primal}) is defined as follows:
	\begin{equation}\label{eq:dual}
		\begin{aligned}
			\max \quad & v\lambda \\
			\text{s.t.} \quad&
			vR\leq uA,\\
			& ue \leq 1,\\
			&u\geq 0,
		\end{aligned}
	\end{equation}
	where $v=(v_1,\ldots,v_I)$ and $u=(u_1,\ldots,u_K)$ are row vectors.
	We next state two lemmas about the dual LP that will be used in our analysis.
	\begin{lemma}\label{lem:pd1}
		Under Assumption~\ref{assu1}, there exists an optimal solution $({v}^*,{u}^*)$ to the dual LP $(\ref{eq:dual})$ satisfying
		\begin{enumerate}
			\item[(i)]  $ \sum_{i=1}^{I}\lambda_iv^*_i=1$.
			\item[(ii)]  $\sum_{k=1}^{K}u^*_k =1$, $u^*_k\geq 0$.
			\item[(iii)]  If  $x^{*}_{ik}> 0$, then $\mu_{ik}v^*_i=u^*_k$.
		\end{enumerate}	
	\end{lemma}
	These properties follow directly from strong duality and the complementary slackness conditions for the primal and dual LPs; the proof is given in Appendix~\ref{sec:pd1_p}. 
	\begin{lemma}\label{lem:pd2}
		If Assumptions~\ref{assu1} and~\ref{assu2} hold, then
		\begin{enumerate}
			\item[(i)]  The optimal dual LP solution $({v}^*,{u}^*)$ is unique. 
			\item[(ii)]  $u^*_k>0, v^*_i>0, \quad   \text{ for all } k\in\mathcal{K}, i\in\mathcal{I}$.
		\end{enumerate}
	\end{lemma}
	\begin{remark}
Lemma \ref{lem:pd2} is crucial in establishing our explicit performance results. Under  Assumptions~\ref{assu1} and~\ref{assu2}, the conclusions of Lemma~\ref{lem:pd2} in the case of a unique optimal processing plan are established in \cite{HarrLope1999}; indeed, Lemma~\ref{lem:pd2}~(i) is one of their equivalent characterizations of the CRP, and Lemma~\ref{lem:pd2}~(ii) follows as a corollary. Our Lemma~\ref{lem:pd2} extends these conclusions to the setting in which the optimal processing plan need not be unique.\end{remark}
\begin{remark}
    Related formulations treat dual uniqueness differently. \cite{Stolyar2005} establishes dual uniqueness through a geometric characterization of the service rate region. \cite{AtaCastRei2024} assume dual uniqueness as a standing condition and, under a decomposable service-rate structure, $\mu_{ik}=\alpha_i\beta_k$ for every $(i,k)\in\mathcal{J}$, obtain an explicit characterization of the unique dual solution.
	\end{remark}
    
      The proof of Lemma~\ref{lem:pd2} uses the communication in Assumption~\ref{assu2} and the equalities in Lemma~\ref{lem:pd1}(iii) along basic activities to express all components of $(v^*,u^*)$ in terms of a single component of $u^*$. The normalization in Lemma~\ref{lem:pd1}(ii) then uniquely determines this component, and hence the entire solution; see Appendix~\ref{sec:pd2_p} for details.  

The components of the dual solution admit a natural workload interpretation: $v^*_i$ can be viewed as the workload contribution associated with a class-$i$ job, while $u^*_k$ can be viewed as the workload contribution associated with server $k$. For a basic activity $(i,k)$, the complementary-slackness relation $\mu_{ik}v^*_i=u^*_k$, or equivalently $v^*_i=m_{ik}u^*_k$, connects these two interpretations: the workload contribution of a class-$i$ job equals its mean processing time at server $k$ multiplied by the workload contribution associated with that server.

     Throughout the paper, we consider systems for which Assumptions~\ref{assu1} and \ref{assu2} hold, i.e., heavy-traffic systems satisfying the relaxed CRP condition. Since the optimal dual solution $(v^*,u^*)$ is unique, we will henceforth use this unique dual solution and omit the superscript ``$*$'' for simplicity.

	In discussing the heavy-traffic regime, we consider a sequence of parallel server systems indexed by $r\in(0,1)$. We keep the service rates fixed and let the arrival-rate vector approach its critical value according to $\lambda_i^{(r)}=\lambda_i(1-r)$, $0<r<1$, for each class $i$. Under this parameterization, the static allocation problem with $\lambda^{(r)}$ replacing $\lambda$ has optimal value $\rho^{(r),*}=1-r$, so that $r$ directly measures the distance from critical loading. Alternative heavy-traffic parameterizations may involve service rates that vary with $r$; such alternative parameterizations can affect the heavy-traffic drift and hence the limiting distribution. Extensions to sequences with varying service rates are left for future investigation.
     
     When $0< r< 1$, the WWTA policy is throughput optimal
	\citep{DaiHarr2020}. Thus, for each $r\in (0,1)$, the queue
	length process, as a continuous-time Markov chain, has a unique
	stationary distribution. We denote by
	${Z}^{(r)}$ the queue length vector having this stationary distribution and call $Z^{(r)}$ the steady-state queue length 
	in the system indexed by $r$. Throughout the paper,     all results are stated in terms of the steady-state queue length vector ${Z}^{(r)}$. We write $f(r)=O(g(r))$ and $f(r)=o(g(r))$
	for, respectively, 
	$$
	\limsup _{r \downarrow 0}\left|\frac{f(r)}{g(r)}\right|<\infty , \quad \lim _{r \downarrow 0} \frac{f(r)}{g(r)}=0.      
	$$
	We now state our main result.
	\begin{theorem}[Limit Distribution for Workload]{\label{thm:main}}
		Suppose that Assumptions \ref{assu1} and \ref{assu2} hold. Under the WWTA policy and any non-idling scheduling policy $\mathcal{P}$, the scaled server-workload vector converges in distribution to a deterministic vector multiplied by a one-dimensional exponential random variable:
		$$\begin{aligned}
			r \Big(W_1\big({Z}^{(r)}\big),\ldots,W_K\big({Z}^{(r)}\big)\Big) \stackrel{d}{\rightarrow}\left(u_1,\ldots,u_K\right) X,\qquad \text{as }r\downarrow0,
		\end{aligned}$$
		where $X$ follows an exponential distribution with mean
		$$\frac{\sum_{i=1}^{I}\lambda_iv_i^2}{\sum_{k=1}^{K}u_k^2}.$$
	\end{theorem}

In Theorem~\ref{thm:main}, the limiting distribution depends only on the arrival rates $\lambda_i$, $i\in\mathcal{I}$, and the unique optimal dual solution $(v,u)$. It does not depend on the specific scheduling policy $\mathcal{P}$. The mean of $X$ can be understood in terms of workload variability and scaling. In particular, half of $\sum_{i=1}^{I}\lambda_i v_i^2$ arises from arrival variability, with $v_i$ representing the workload contribution of a class-$i$ job. The other half arises from service variability; although service fluctuations are activity dependent, the transform-level flow-balance identity aggregates their contributions across servers and yields the same class-level expression. The appearance of $(u_1,\ldots,u_K)$ in the limiting workload vector reflects the state-space collapse direction: server $k$ carries a limiting workload proportional to $u_k$, and the denominator $\sum_{k=1}^{K}u_k^2$ arises from the quadratic scaling along this collapse direction, determining the normalization of the common random factor $X$. These interpretations will become explicit in the analysis that follows.

The proof of Theorem \ref{thm:main} is provided in Section \ref{sec:mainproof}. The key ingredient for proving Theorem~\ref{thm:main} is \textit{state-space collapse}, which will be presented in Section \ref{sec:ssc_moment}. Other supporting results will be presented in Sections \ref{sec:outline1} and \ref{sec:outline2}, with the latter relying on the state-space collapse in Section \ref{sec:ssc_moment}.

	\section{Preliminary Results I}\label{sec:outline1}
In this section, we introduce the main technical tool, the BAR approach. We also establish preliminary idleness and flow-balance relations that will be used throughout the analysis. Recall the WWTA routing policy (\ref{eq:wwta}). Our results do not depend on the choice of tie-breaking rule. For example, we can assume that a job is routed to the server with the smallest index if more than one server achieves the minimum. With this tie-breaking convention, we denote the resulting minimizer by
	$$L^{(i)}(z)=\argmin_{k\in \mathcal{K}(i)}m_{ik}W_k({z}). $$
	 It is clear that, for each class $i\in \mathcal{I}$, 
	$$\sum_{k\in \mathcal{K}(i)}\mathbbm{1}\left(k=L^{(i)}({z})\right) =1.$$
    Recall that, under a generic non-idling scheduling policy $\mathcal P$, $P_{ik}(z)$ denotes the fraction of server $k$'s effort devoted to serving class-$i$ jobs in state $z$. For each $r\in(0,1)$, under the WWTA routing policy and a non-idling scheduling policy $\mathcal P$, the queue-length process is a continuous-time Markov chain with the \emph{generator}
	\begin{equation}\label{eq:generator}
		\begin{aligned}
			Gf({z})=
			&\sum_{i=1}^{I}\lambda^{(r)}_i\sum_{k\in \mathcal{K}(i)}\Big[f({z}+{e}_{ik})-f({z})\Big]\mathbbm{1}\left(k=L^{(i)}({z})\right)\\
			+&\sum_{k=1}^{K}\sum_{i\in \mathcal{I}(k)}\mu_{ik}P_{ik}({z})\Big[f({z}-{e}_{ik})-f({z})\Big], \quad z\in \Z_+^J,
	\end{aligned}\end{equation}
	where $f:\Z^J_+\to \R$, and ${e}_{ik}\triangleq {e}_j =(0,\ldots,1,\ldots,0)^\top\in
	\mathbb{R}^J$ is the unit vector whose $j$th element is one and whose other elements are zero. 
    
    Although our results concern server workloads, the workload vector alone does not fully specify the system dynamics: the service-completion rates depend on the job class composition at each server and on the activity-specific service allocation $P_{ik}(z)$. As a result, the queue length $z$ is used to formulate the generator and, consequently, the BAR. Despite the explicit dependence on $P_{ik}(z)$, the scheduling-specific terms will disappear when the service terms are aggregated at the server level: by the non-idling property (\ref{eq:nonidle}), the resulting aggregate service allocation reduces to an indicator of whether the server is idle.
    
    The following lemma states the BAR that will be used throughout the paper.
	
	\begin{lemma}\label{lem:steady_exp}
		Under Assumptions \ref{assu1} and \ref{assu2}, with the WWTA routing policy and a non-idling scheduling policy $\mathcal{P}$, 
		let $f: \mathbb{Z}^J_+ \rightarrow \mathbb{R}$ be a function. Suppose there exist $n\in \mathbb{N}^+$ and some constant $C>0$ such that  $|f({z})| \leq C\big(1+\sum_{k=1}^{K}W^{n}_k(z)\big)$ for $z\in \Z_+^J$ (i.e., $f({z})$ is dominated by a polynomial function of workloads).  Then, for each $0<r<1$, the steady-state queue length vector ${Z}^{(r)}$ satisfies the following basic adjoint relationship (BAR):
		\begin{align}\label{eq:BAR}
			\E\Big[Gf\Big({Z}^{(r)}\Big)\Big]=0.                  
		\end{align}
	\end{lemma}
	    A function $f$ for which (\ref{eq:BAR}) holds is called a test function. It was shown in 
	\cite{GlynZeev2008} that every bounded function
	$f: \mathbb{Z}^J_+ \rightarrow \mathbb{R}$ is a test function.
	The proof of Lemma
	\ref{lem:steady_exp} is provided in Appendix
	\ref{sec:steady}, using an argument similar to that in \cite{BravDaiFeng2016}.
	
	For notational convenience in stating the preliminary results throughout the paper, we define the dual slack in (\ref{eq:dual}) for activity $(i,k)$ by
	\begin{align}\label{eq:dik}
		d_{ik} = u_k-\mu_{ik}v_i.
	\end{align}
	By dual feasibility, $d_{ik}\geq 0$. Complementary slackness implies that $d_{ik}>0$ only if activity $(i,k)$ is non-basic. Conversely, the strict complementary slackness theorem, together with the uniqueness of the optimal dual solution established in Lemma~\ref{lem:pd2}, implies that every non-basic activity satisfies $d_{ik}>0$. Thus, an activity $(i,k)$ is non-basic if and only if $d_{ik}>0$.
	\begin{lemma}[Idleness and flow balance]\label{lem:idle}
		Under Assumptions \ref{assu1} and \ref{assu2}, with the WWTA routing policy and a non-idling scheduling policy $\mathcal{P}$, for each $r\in (0,1)$,
		\begin{align}
			& \sum_{k=1}^K u_k \Prob\Big(\sum_{i\in \mathcal{I}(k)}Z^{(r)}_{ik}=0\Big)+ \sum_{k=1}^{K}\sum_{i\in \mathcal{I}(k)}d_{ik}\E\left[P_{ik}({Z^{(r)}})\right]= r, \label{eq:idle}\\
			& \lambda_i^{(r)}\Prob\Big(k=L^{(i)}\big({Z}^{(r)}\big)\Big)=\mu_{ik}\E\Big[P_{ik}\Big({Z}^{(r)}\Big)\Big], \quad \text{ for each }(i,k)\in\mathcal J.\label{eq:flowbalance}
		\end{align}                  
	\end{lemma}
	Equation (\ref{eq:idle}) implies that each server's idle
	probability is $O(r)$. This means that the WWTA
	routing policy efficiently drives the utilization of each server to $100\%$ in heavy traffic. Equation
	(\ref{eq:idle}) also implies that
	\begin{equation}\label{eq:nbcon_0}
	\begin{aligned}
		\E\left[P_{ik}({Z^{(r)}})\right] =O(r),  \quad\text{if }  d_{ik}>0,
	\end{aligned}
	\end{equation}
	which means that, for such a non-basic activity, the long-run fraction of time that server $k$ spends processing class $i$ jobs is negligible (of order $O(r)$). Equation~(\ref{eq:flowbalance}) gives the usual flow balance for buffer $(i,k)$: the rate into the buffer equals the rate out of the buffer. Combining (\ref{eq:nbcon_0}) with (\ref{eq:flowbalance}) shows that for a non-basic activity, the probability that a class-$i$ job is routed by WWTA to server $k$ is also negligible, namely,
	\begin{equation}\label{eq:nbcon_1}
		\Prob\Big(k=L^{(i)}\big({Z}^{(r)}\big)\Big)= O(r),  \quad \text{if } d_{ik}>0.
	\end{equation}
	
	\begin{proof}[Proof of Lemma~\ref{lem:idle}]
		To prove (\ref{eq:idle}),
		we let $f({z})=\sum_{i=1}^I\sum_{k\in \mathcal{K}(i)}v_iz_{ik}$. Applying the generator (\ref{eq:generator}) to this test function gives
		$$\begin{aligned}
			Gf({z})
			=&\sum_{i=1}^{I}\lambda^{(r)}_iv_i-\sum_{k=1}^{K}\sum_{i\in \mathcal{I}(k)}\mu_{ik}v_iP_{ik}({z})\\
			= &-r +\sum_{k=1}^{K}u_k\mathbbm{1}\Big(\sum_{i\in \mathcal{I}(k)}z_{ik}=0\Big)+\sum_{k=1}^{K}\sum_{i\in \mathcal{I}(k)}d_{ik}P_{ik}({z}),
		\end{aligned}$$
		where the last equality follows from Lemma \ref{lem:pd1} and $(\ref{eq:dik})$. Then  (\ref{eq:idle}) follows from (\ref{eq:BAR}).
		
		To prove (\ref{eq:flowbalance}), 
		we take $f({z})=z_{ik}$.
		Then,
		$$\begin{aligned}
			Gf({z})=&\lambda_i^{(r)}\mathbbm{1}\left(k=L^{(i)}({z})\right)-\mu_{ik}P_{ik}({z}),
		\end{aligned}$$
		and (\ref{eq:BAR}) implies (\ref{eq:flowbalance}).
	\end{proof}
	Throughout the paper, we adopt the convention:
	for a sequence $\{a(i)\}$,     
	$$    
	\begin{aligned}
		\sum_{i\in \mathcal{I}(k)}^b a(i) = \sum_{i\in \mathcal{I}(k): d_{ik}=0} a(i),  \quad \sum_{i\in \mathcal{I}(k)}^{nb} a(i) =\sum_{i\in \mathcal{I}(k): d_{ik}>0} a(i).         
	\end{aligned}$$
	The first sum ranges over all basic activities, while the second ranges over all non-basic activities. This convention is useful because non-basic activities, characterized by $d_{ik}>0$, require separate treatment in the analysis that follows. The same superscripts are used analogously for sums over $k\in\mathcal{K}(i)$.
	
	\section{State-Space Collapse}\label{sec:ssc_moment}
	In this section, we establish a state-space collapse result that plays a critical role in proving Theorem~\ref{thm:main}. Recall that $u\in\R^K$ is the server component of the optimal dual solution.  For each $k\in\mathcal{K}$, we define
	$$T_k({z})=\frac{1}{u_k}W_k({z}),$$ 
	and denote by $T_{(k)}({z})$ the $k$th smallest value among $\{T_k({z}): k\in\mathcal{K}\}$ given the queue length $z$.  
    The construction of $T_k(z)$ is motivated by the nature of WWTA routing, under which a class-$i$ job is routed by comparing $m_{ik}W_k(z)$ across $k\in\mathcal K(i)$. For a basic activity $(i,k)$, the relation $\mu_{ik}v_i=u_k$ gives
$$\frac{m_{ik}W_k(z)}{v_i}=\frac{W_k(z)}{u_k}=T_k(z).$$
Thus, among the basic activities for a fixed class $i$, WWTA can be viewed as comparing the corresponding $T_k(z)$ values. The non-basic activities are exceptions, but they are expected to play a diminishing role in heavy traffic, consistent with Lemma~\ref{lem:idle}. Together, these observations suggest using $T_k(z)$ values as a natural common scale for comparing workloads across heterogeneous servers, and our state-space collapse result shows that their range remains bounded in heavy traffic.
	\begin{proposition}[State-space collapse]{\label{prop:sscmoment}}
		Suppose that Assumptions \ref{assu1} and \ref{assu2} hold. Under
		the WWTA policy and any non-idling scheduling policy $\mathcal{P}$, for each $n\in \mathbb{N}^+$,
		there exists a constant $M_n>0$ such that
		\begin{align}
			\label{eq:ssc}
			\E\bigg[\Big(T_{(K)}\big({Z}^{(r)}\big)-T_{(1)}\big({Z}^{(r)}\big)\Big)^n\bigg]\leq M_n, \quad \text{ for } r\in (0, r_0),
		\end{align}            
		where $r_0\in(0,1)$ is some constant independent of $n$.
	\end{proposition}	
Proposition~\ref{prop:sscmoment} shows that, after accounting for server heterogeneity through the weights $u_k$, the server workloads remain close on this common scale as the overall workload grows in heavy traffic. Notably, Proposition~\ref{prop:sscmoment} establishes moment bounds of all orders, although only moments up to the second order are needed to establish the main results in Theorem~\ref{thm:main}.

A geometric view of this state-space collapse is to measure the deviation of the workload vector from the one-dimensional direction spanned by $u$. Inspired by \cite{EryiSrik2012}, we introduce the following workload decomposition for constructing the test function used to prove Proposition~\ref{prop:sscmoment}. Denote by $\langle\cdot,\cdot\rangle$ the inner product and by $\|\cdot\|$ the $\ell_2$ norm on $\mathbb R^K$. Define the unit column vector in the direction of $u$ by 
$$\hat{u}=\frac{u^\top}{\|u\|}=\left(\frac{u_1}{\|u\|},\ldots,\frac{u_K}{\|u\|}\right)^\top.$$ 
The projection of the workload vector $W(z)$ onto the direction spanned by $\hat u$ can be written as follows
$$W_{\parallel}(z)\triangleq\langle W(z),\hat{u}\rangle\hat{u}=\left(\frac{1}{\|u\|^2}\sum_{\ell=1}^{K}u_{\ell}W_{\ell}(z)\right)\begin{pmatrix}u_1\\\vdots\\u_K\end{pmatrix}.$$
We then define the perpendicular component of the workload vector as
$$W_{\perp}(z)\triangleq W(z)-W_{\parallel}(z),$$
and denote its $k$th component by $W_{\perp,k}(z)$. 
The norm of the perpendicular component and the workload range in Proposition~\ref{prop:sscmoment} satisfy
$$
\frac{\min_{k\in\mathcal{K}}u_k}{2}
\big[T_{(K)}(z)-T_{(1)}(z)\big]
\leq \|W_{\perp}(z)\|
\leq T_{(K)}(z)-T_{(1)}(z).
$$
Thus, controlling the perpendicular component is equivalent, up to constant factors, to controlling the workload range in Proposition~\ref{prop:sscmoment}. Motivated by this relation, to establish the $n$th-moment bound, we construct the test function
\begin{equation}\label{eq:ssctest}
f(z)=\|W_{\perp}(z)\|^{n+1}=\Bigg\{\sum_{k=1}^K\bigg[W_k(z)-\frac{u_k}{\|u\|^2}\sum_{\ell=1}^{K}u_{\ell}W_{\ell}(z)\bigg]^2\Bigg\}^{\frac{n+1}{2}}.
\end{equation}
The main idea in proving Proposition \ref{prop:sscmoment} is to apply the generator to the test function in (\ref{eq:ssctest}) and use the BAR to derive a recursive bound on the moments of the range $T_{(K)}(z)-T_{(1)}(z)$. In particular, we bound its $n$th moment in terms of lower-order moments, allowing the desired bounds to be established by induction. To carry out this argument, we first present two lemmas. The first lemma bounds the increments of the test function.
	\begin{lemma}\label{lem:ssc1}
		    For the test function in $(\ref{eq:ssctest})$, we have the following increment bounds:
		$$\begin{aligned}
			&f(z+e_{ik})-f(z) \leq (n+1)m_{ik}W_{\perp,k}(z) \|W_{\perp}(z)\|^{n-1} + \sum_{\ell=0}^{n-1}c_{\ell}^A\|W_{\perp}(z)\|^{\ell},\\
			&f(z-e_{ik})-f(z) \leq -(n+1)m_{ik}W_{\perp,k}(z) \|W_{\perp}(z)\|^{n-1} + \sum_{\ell=0}^{n-1}c_{\ell}^S\|W_{\perp}(z)\|^{\ell},\\
		\end{aligned}$$
		where $c_{\ell}^A$ and $c_{\ell}^S$ denote the constants associated with arrival and service terms, respectively. 
	\end{lemma}
   The second lemma relates the leading terms involving the perpendicular workload components to the range $T_{(K)}(z)-T_{(1)}(z)$.
    	\begin{lemma}\label{lem:ssc2}
		There exists a collection $\{\lambda_{ik}\}$ satisfying $\lambda_{i}=\sum_{k\in \mathcal{K}(i)}\lambda_{ik}$ for each $i\in\mathcal I$ and $\sum_{i\in \mathcal{I}(k)}\lambda_{ik}m_{ik} =1$     for each $k\in\mathcal K$, such that
		\begin{align}
			&\sum_{i=1}^{I}\lambda_i\sum_{k\in \mathcal{K}(i)}m_{ik}W_{\perp,k}(z)\mathbbm{1}\big(k=L^{(i)}({z})\big)
			-\sum_{k=1}^{K}W_{\perp,k}(z)\leq -\min_{(i,k):\,\lambda_{ik}>0} v_i \lambda_{ik} \big[T_{(K)} (z)- T_{(1)}(z)\big],	\label{eq:ub1}\\
			&-r\sum_{i=1}^{I}\lambda_i\sum_{k\in \mathcal{K}(i)}m_{ik}W_{\perp,k}(z)\mathbbm{1}\big(k=L^{(i)}({z})\big)\leq r\max_{(i,k)}\frac{u_k}{u_k-d_{ik}}\big[T_{(K)} (z)- T_{(1)}(z)\big]. \label{eq:ub2}
		\end{align}
	\end{lemma}
    In what follows, $(i,k)$ in a minimum or maximum is understood to range over $(i,k)\in\mathcal J$.
	The proofs of Lemmas~\ref{lem:ssc1} and~\ref{lem:ssc2} are given in Appendices~\ref{sec:ssc_lm1} and~\ref{sec:ssc_lm2}, respectively.

	\begin{proof}[Proof of Proposition \ref{prop:sscmoment}]
      Applying the generator to the test function (\ref{eq:ssctest}), and using the increment bounds in Lemma~\ref{lem:ssc1}, we obtain
	$$\begin{aligned}
		Gf(z)
		=&\sum_{i=1}^{I}\lambda^{(r)}_i\sum_{k\in \mathcal{K}(i)}\Big[f({z}+{e}_{ik})-f({z})\Big]\mathbbm{1}\left(k=L^{(i)}({z})\right)
		+\sum_{k=1}^{K}\sum_{i\in \mathcal{I}(k)}\mu_{ik}P_{ik}({z})\Big[f({z}-{e}_{ik})-f({z})\Big]\\
		\leq& (n+1)\bigg[\sum_{i=1}^I\lambda_i^{(r)}\sum_{k\in\mathcal{K}(i)}m_{ik}W_{\perp,k}(z)\mathbbm{1}\left(k=L^{(i)}({z})\right) - \sum_{k=1}^K\sum_{i\in\mathcal{I}(k)}P_{ik}({z})W_{\perp,k}(z) \bigg] \|W_{\perp}(z)\|^{n-1} \\
&+\sum_{i=1}^I\lambda_i^{(r)}\sum_{k\in\mathcal{K}(i)}\bigg[\sum_{\ell=0}^{n-1}c_{\ell}^A\|W_{\perp}(z)\|^{\ell}\bigg]\mathbbm{1}\left(k=L^{(i)}({z})\right) + \sum_{k=1}^K\sum_{i\in\mathcal{I}(k)}\bigg[\sum_{\ell=0}^{n-1}c_{\ell}^S\|W_{\perp}(z)\|^{\ell}\bigg]\\
		\leq& (n+1)\bigg[\sum_{i=1}^I\lambda_i\sum_{k\in\mathcal{K}(i)}m_{ik}W_{\perp,k}(z)\mathbbm{1}\left(k=L^{(i)}({z})\right) - \sum_{k=1}^KW_{\perp,k} (z)\bigg] \|W_{\perp}(z)\|^{n-1} \\
		&-r(n+1)\|W_{\perp}(z)\|^{n-1}\sum_{i=1}^I\lambda_i\sum_{k\in\mathcal{K}(i)}m_{ik}W_{\perp,k}(z)\mathbbm{1}\left(k=L^{(i)}({z})\right)+\sum_{\ell=0}^{n-1}c_{\ell}\|W_{\perp}(z)\|^{\ell},
	\end{aligned}$$
	where $c_{\ell} \triangleq c_{\ell}^A\sum_{i=1}^I\lambda_i + I\cdot K\cdot c_{\ell}^S$. In the last inequality, we use
	$$\begin{aligned}-\sum_{i\in\mathcal{I}(k)}P_{ik}({z})W_{\perp,k}(z) =-W_{\perp,k}(z) + W_{\perp,k}(z)\mathbbm{1} \bigg(\sum_{i\in
			\mathcal{I}(k)}z_{ik}=0\bigg)\leq  -W_{\perp,k}(z).\end{aligned}$$
Applying Lemma~\ref{lem:ssc2} to the two leading terms in the generator gives
\begin{equation}\label{eq:gen_ub2}
\begin{aligned}
Gf(z)
\leq& -(n+1)\Big(\min_{(i,k):\,\lambda_{ik}>0} v_i\lambda_{ik}
-r\max_{(i,k)}\frac{u_k}{u_k-d_{ik}}\Big)
\big[T_{(K)}(z)-T_{(1)}(z)\big]\|W_{\perp}(z)\|^{n-1}\\
&+\sum_{\ell=0}^{n-1}c_{\ell}\|W_{\perp}(z)\|^{\ell}.
\end{aligned}
\end{equation}
Choose $r_0\in(0,1)$ sufficiently small such that
$$\min_{(i,k):\,\lambda_{ik}>0} v_i\lambda_{ik}
-r_0\max_{(i,k)}\frac{u_k}{u_k-d_{ik}}>0.$$
Then, for all $r\in(0,r_0)$,
$$\min_{(i,k):\,\lambda_{ik}>0} v_i\lambda_{ik}
-r\max_{(i,k)}\frac{u_k}{u_k-d_{ik}}
\geq
\min_{(i,k):\,\lambda_{ik}>0} v_i\lambda_{ik}
-r_0\max_{(i,k)}\frac{u_k}{u_k-d_{ik}}>0.$$
We next relate $\|W_{\perp}(z)\|$ to $T_{(K)}(z)-T_{(1)}(z)$. Note that $$\frac{1}{\|u\|^2}\sum_{\ell\in\mathcal{K}}u_{\ell}W_{\ell}(z)=\frac{1}{\|u\|^2}\sum_{\ell\in\mathcal{K}}u^2_{\ell}T_{\ell}(z)$$ is the weighted average of the $T_k(z)$ values and hence lies between $T_{(1)}(z)$ and $T_{(K)}(z)$. Therefore, $$\|W_{\perp}(z)\| \leq\sqrt{\sum_{k\in\mathcal{K}}u_k^2\big[T_{(K)}(z)-T_{(1)}(z)\big]^2} \leq T_{(K)}(z)-T_{(1)}(z),$$ where the last inequality follows from $\sum_{k\in\mathcal{K}}u_k^2\leq1$, since $\sum_{k\in\mathcal{K}}u_k=1$ and $u_k\geq0$. On the other hand,
$$\begin{aligned}
\|W_{\perp}(z)\|\geq&\sqrt{\min_{k\in\mathcal{K}}u_k^2\sum_{k\in\mathcal{K}}\bigg(T_k(z)-\frac{1}{\|u\|^2}\sum_{\ell\in\mathcal{K}}u_{\ell}W_{\ell}(z)\bigg)^2}\\
\geq&\min_{k\in\mathcal{K}}u_k\sqrt{\bigg(T_{(K)}(z)-\frac{1}{\|u\|^2}\sum_{\ell\in\mathcal{K}}u_{\ell}W_{\ell}(z)\bigg)^2 + \bigg(T_{(1)}(z)-\frac{1}{\|u\|^2}\sum_{\ell\in\mathcal{K}}u_{\ell}W_{\ell}(z)\bigg)^2}\\
\geq&\frac{\min_{k\in\mathcal{K}}u_k}{2}\big[T_{(K)}(z)-T_{(1)}(z)\big].\end{aligned}$$
Using the relation between $\|W_{\perp}(z)\|$ and $T_{(K)}(z)-T_{(1)}(z)$ established above, (\ref{eq:gen_ub2}) implies
\begin{equation}\label{eq:gen_ub3}
\begin{aligned}
Gf(z)
\leq&-\frac{n+1}{2^{n-1}}(\min_{k\in\mathcal{K}}u_k)^{n-1}
\bigg(\min_{(i,k):\,\lambda_{ik}>0} v_i\lambda_{ik}
-r_0\max_{(i,k)}\frac{u_k}{u_k-d_{ik}}\bigg)
\big[T_{(K)}(z)-T_{(1)}(z)\big]^n\\
&+\sum_{\ell=0}^{n-1}c_{\ell}\big[T_{(K)}(z)-T_{(1)}(z)\big]^{\ell}.
\end{aligned}
\end{equation}
Taking steady-state expectations on both sides of (\ref{eq:gen_ub3}) and applying Lemma~\ref{lem:steady_exp}, we obtain the following bound for all $r\in(0,r_0)$:
$$
\E\bigg[\Big(T_{(K)}\big(Z^{(r)}\big)-T_{(1)}\big(Z^{(r)}\big)\Big)^n\bigg]
\leq
\frac{2^{n-1}\sum_{\ell=0}^{n-1}c_{\ell}
\E\big[\big(T_{(K)}(Z^{(r)})-T_{(1)}(Z^{(r)})\big)^{\ell}\big]}
{(n+1)(\min_{k\in\mathcal{K}}u_k)^{n-1}
\big(\min_{(i,k):\,\lambda_{ik}>0} v_i\lambda_{ik}
-r_0\max_{(i,k)}\frac{u_k}{u_k-d_{ik}}\big)}.
$$
For $n=1$, the right-hand side is constant. The desired moment bound then follows for each $n\in\mathbb{N}^+$ by induction. This completes the proof.
	\end{proof}	
	
Here, our proof establishes polynomial moment bounds recursively through an induction on the moment order. The projection-based formulation in \cite{EryiSrik2012}, which measures the deviation from the collapse direction through the perpendicular projection of the queue length vector, motivates the construction of our test function. The two proofs proceed differently: \cite{EryiSrik2012} verify a negative-drift condition together with uniformly bounded increments for the norm of the perpendicular component and then apply the drift result of \cite{Hajek1982} to obtain an exponential-moment bound, whereas our proof recursively controls each higher-order moment in terms of lower-order moments.
	
\section{Preliminary Results II}\label{sec:outline2}
	In this section, we establish a series of preliminary results concerning asymptotic negligibility and workload moment bounds as the system approaches heavy traffic. These results build on the state-space collapse established in Proposition~\ref{prop:sscmoment}. As noted earlier, whenever Proposition~\ref{prop:sscmoment} is invoked in the subsequent analysis, only the cases $n=1$ and $n=2$, corresponding to the first and second moments, are needed.
    
    We first quantify the expected total workload on the event that a given server is idle. By state-space collapse, when one server has zero workload, the workloads of the other servers should be small as well. Together with the $O(r)$ idle probability established in Lemma~\ref{lem:idle}, this yields the following result.
	\begin{lemma}\label{lem:neg}
		Under Assumptions \ref{assu1} and \ref{assu2}, with the WWTA policy and a non-idling scheduling policy $\mathcal{P}$, for each server $k\in \mathcal{K}$ and $r\in(0,r_0)$,
		$$\begin{aligned}
			\quad\E\bigg[\bigg(\sum_{k'=1}^Ku_{k'}W_{k'}\Big({Z}^{(r)}\Big)\bigg)\mathbbm{1}\Big(\sum_{i\in \mathcal{I}(k)}Z^{(r)}_{ik}=0\Big)\bigg]=O(r^{1/2}).
		\end{aligned}$$
	\end{lemma}
	\begin{proof}
		$$\begin{aligned}
			&\E\bigg[\bigg(\sum_{k'=1}^Ku_{k'}W_{k'}({Z}^{(r)})\bigg)\mathbbm{1}\bigg(\sum_{i\in \mathcal{I}(k)}Z^{(r)}_{ik}=0\bigg)\bigg]\\
			=&\E\bigg[\bigg|\sum_{k'=1}^Ku_{k'}\sum_{i'\in \mathcal{I}(k')}\frac{1}{\mu_{i'k'}}Z^{(r)}_{i'k'}-\sum_{k'=1}^Ku^2_{k'}\frac{1}{u_k}\sum_{i\in \mathcal{I}(k)}\frac{1}{\mu_{ik}}Z^{(r)}_{ik}\bigg|\mathbbm{1}\bigg(\sum_{i\in \mathcal{I}(k)}Z^{(r)}_{ik}=0\bigg)\bigg]\\
			\leq &\sum_{k'=1}^Ku^2_{k'}\E\bigg[\bigg|\frac{1}{u_{k'}}\sum_{i'\in \mathcal{I}(k')}\frac{1}{\mu_{i'k'}}Z^{(r)}_{i'k'}-\frac{1}{u_k}\sum_{i\in \mathcal{I}(k)}\frac{1}{\mu_{ik}}Z^{(r)}_{ik}\bigg|\mathbbm{1}\bigg(\sum_{i\in \mathcal{I}(k)}Z^{(r)}_{ik}=0\bigg)\bigg]\\
			\stackrel{(a)}{\leq}&\E\bigg[\bigg(T_{(K)}({Z}^{(r)})-T_{(1)}({Z}^{(r)})\bigg)\mathbbm{1}\bigg(\sum_{i\in \mathcal{I}(k)}Z^{(r)}_{ik}=0\bigg)\bigg]\\
			\stackrel{(b)}{\leq}&\E\bigg[\bigg(T_{(K)}({Z}^{(r)})-T_{(1)}({Z}^{(r)})\bigg)^2\bigg]^{1/2}\Prob\bigg(\sum_{i\in \mathcal{I}(k)}Z^{(r)}_{ik}=0\bigg)^{1/2}\\
			\stackrel{(c)}{\leq}&M_2^{1/2}\Prob\bigg(\sum_{i\in \mathcal{I}(k)}Z^{(r)}_{ik}=0\bigg)^{1/2}\stackrel{(d)}{=}O(r^{1/2}),
		\end{aligned}$$
		where $(a)$ follows from $\sum_{k=1}^K u_k^2\leq1$, $(b)$ is by the Cauchy--Schwarz inequality, $(c)$ is by Proposition \ref{prop:sscmoment}, and $(d)$ follows from equation~(\ref{eq:idle}).
	\end{proof}
	Recall that, in Lemma~\ref{lem:idle}, we obtained negligibility results for non-basic activities, namely (\ref{eq:nbcon_0}) and (\ref{eq:nbcon_1}), which show that both the probability of routing to a non-basic activity and the expected fraction of time spent processing such an activity are $O(r)$. The following lemma further shows that this negligibility persists when the routing and service terms are multiplied by the weighted sum of workloads. It also shows that the expected weighted number of jobs associated with non-basic activities is negligible under the heavy-traffic scaling.

	\begin{lemma}[Negligibility for non-basic activities]{\label{lem:nbcon}}
		Under Assumptions \ref{assu1} and \ref{assu2}, with the WWTA policy and a non-idling scheduling policy $\mathcal{P}$, the following bounds hold for each activity $(i,k)\in\mathcal J$ and $r\in(0,r_0)$:
		\begin{equation}\label{eq:nbcon_2}
			\E\left[\bigg(\sum_{k'=1}^Ku_{k'}W_{k'}\Big({Z}^{(r)}\Big)\bigg)\mathbbm{1}\Big(k=L^{(i)}\big({Z}^{(r)}\big)\Big)\right]=O(r^{1/2}),\quad \text{if } d_{ik}>0,
		\end{equation}
		\begin{equation}\label{eq:nbcon_3}\E\left[P_{ik}\Big({Z}^{(r)}\Big)\bigg(\sum_{k'=1}^Ku_{k'}W_{k'}\Big({Z}^{(r)}\Big)\bigg)\right]=O(r^{1/2}),\quad \text{if } d_{ik}>0.
		\end{equation}
		In addition, the scaled first moment of the weighted sum of jobs associated with non-basic activities is negligible: 
		\begin{equation}\label{eq:nbcon_4}
			r\E\left[\sum_{i=1}^I\sum_{k\in \mathcal{K}{(i)}}^{nb}v_iZ^{(r)}_{ik}\right]=O(r^{1/2}).
		\end{equation} 
	\end{lemma}
	The proof of Lemma \ref{lem:nbcon} is provided in Appendix \ref{sec:nbcon}.
	Finally, we establish a first-moment bound for the scaled weighted sum of queue lengths. This bound will be used in proving the transform-level flow-balance relation in Section~\ref{sec:ld_ingredient2}.
	\begin{lemma}[First-moment bound]{\label{lem:mbound1}}
Under the conditions of Lemma~\ref{lem:nbcon}, there exists a constant $M>0$ such that, for $r\in(0,r_0)$,
		$$
r\E\left[\sum_{i=1}^I\sum_{k\in \mathcal{K}{(i)}}v_iZ^{(r)}_{ik}\right]\leq M.
$$
	\end{lemma}
	The proof of Lemma \ref{lem:mbound1} is provided in Appendix \ref{sec:mbound1}.

\section{Proofs for Theorem \ref{thm:main}}\label{sec:mainproof}
	The proof proceeds by first identifying the limiting distribution of a weighted scalar quantity through an exponential test function and then combining this result with the state-space-collapse result in Proposition~\ref{prop:sscmoment} to obtain the joint limit of the server workload vector. We first state Proposition~\ref{prop:ld} and Corollary~\ref{coro:ld_workload}, where the scalar quantity is represented as the weighted sum of queue lengths and the weighted sum of workloads, respectively. The remainder of this section is organized as follows. Sections~\ref{sec:ld_ingredient1} and~\ref{sec:ld_ingredient2} establish two key lemmas used to prove Proposition~\ref{prop:ld}. Section~\ref{sec:ld} proves Proposition~\ref{prop:ld}, followed by the proof of the main result, Theorem~\ref{thm:main}, in Section~\ref{sec:mdgwh_c}. 
	
	We construct the exponential test function $f_{\theta}: \mathbb{Z}^J_+ \rightarrow \mathbb{R}$ by
	\begin{equation}\label{eq:ef} f_{\theta}(z) = e^{r\theta\left(\sum_{i=1}^I\sum_{k\in \mathcal{K}(i)}v_iz_{ik}\right)}, \quad \theta\leq 0.\end{equation}
	Then we define the corresponding Laplace transform $\phi^{(r)}$, together with two related quantities, as follows:
	\begin{equation}\label{eq:laplace} 
		\begin{aligned}
			& \phi^{(r)}(\theta) = \E\left[f_{\theta}(Z^{(r)})\right],\quad \theta\leq 0,\\
			& \phi^{(r)}_{k}(\theta) = \E\bigg[f_{\theta}(Z^{(r)})\mathbbm{1}\bigg(\sum_{i\in \mathcal{I}(k)}Z^{(r)}_{ik}=0\bigg)\bigg],\quad k\in\mathcal{K},\\
			& \phi^{(r)}_{ik}(\theta) = \E\left[P_{ik}({Z^{(r)}})f_{\theta}(Z^{(r)})\right], \quad (i,k)\in\mathcal{J}.
	\end{aligned}\end{equation}
    Here, $\phi^{(r)}_{k}(\theta)$ is the exponentially weighted idleness quantity associated with server $k$, while $\phi^{(r)}_{ik}(\theta)$ is the exponentially weighted service-allocation quantity associated with activity $(i,k)$. These quantities arise naturally from the idleness and activity-specific service terms, respectively, when the BAR is applied to the exponential test function.
	\begin{proposition}[Limit Distribution for the Weighted Queue Length]{\label{prop:ld}}
		Suppose that Assumptions \ref{assu1} and \ref{assu2} hold.  Under the WWTA policy and a non-idling scheduling policy $\mathcal{P}$, for each $\theta\leq 0$, we have
		$$\lim\limits_{r\downarrow 0}\phi^{(r)}(\theta) 
			=\frac{1}{1-\theta\sum_{i=1}^{I}\lambda_iv_i^2}.
		$$
		That is, the limit is the Laplace transform of an exponential random variable with mean 
		$m = \sum_{i=1}^{I}\lambda_iv_i^2$.
		Therefore, if we denote by $\tilde{X}$ an exponential random variable with mean $m$, the scaled weighted sum of queue lengths converges in distribution to $\tilde{X}$:   $$r\bigg(\sum_{i=1}^I\sum_{k\in \mathcal{K}(i)}v_iZ^{(r)}_{ik}\bigg)\stackrel{d}{\rightarrow} \tilde{X},\quad \text{ as }r \downarrow 0.$$
	\end{proposition}
	\begin{corollary}[Limit Distribution for the Weighted Workload]\label{coro:ld_workload}
		Under the conditions of Proposition \ref{prop:ld}, the scaled weighted sum of server workloads converges in distribution to the random variable $\tilde{X}$:
		$$r\bigg(\sum_{k=1}^{K}u_kW_k\Big({Z}^{(r)} \Big)\bigg)\stackrel{d}{\rightarrow} \tilde{X},\quad \text{ as }r \downarrow 0.$$
	\end{corollary}
	
	The passage from Proposition~\ref{prop:ld} to Corollary~\ref{coro:ld_workload} is nontrivial. It is straightforward if $d_{ik}=0$ for all activities, which in particular holds if all activities are basic. In general,
		\begin{equation}\label{eq:relation}
			\sum_{i=1}^I\sum_{k\in \mathcal{K}(i)}v_iZ^{(r)}_{ik}=\sum_{k=1}^K\sum_{i\in \mathcal{I}(k)}\frac{u_k-d_{ik}}{\mu_{ik}}Z^{(r)}_{ik}=\sum_{k=1}^Ku_kW_k\Big({Z}^{(r)}\Big)-\sum_{k=1}^K\sum_{i\in \mathcal{I}(k)}\frac{d_{ik}}{\mu_{ik}}Z^{(r)}_{ik},
		\end{equation}
		where the additional term due to non-basic activities with $d_{ik}>0$ requires a separate treatment. The proof of Corollary~\ref{coro:ld_workload} is provided in Appendix~\ref{sec:ld_workload}. 
\begin{remark}
    Corollary~\ref{coro:ld_workload} provides the workload counterpart of Proposition~\ref{prop:ld}, which will later be combined with the state-space-collapse result to derive the joint convergence. In principle, the weighted sum of workloads could also
    be used directly as the scalar quantity in \eqref{eq:ef}. We choose to work with the queue-length representation because the generator and BAR are formulated in terms of the queue-length vector. Working directly with the workload would require repeatedly translating back to the queue-length representation and handling the additional term in \eqref{eq:relation}. We therefore work with the queue-length representation throughout and translate to the workload only once at the end.
\end{remark}

	\subsection{Ingredient I for Proposition \ref{prop:ld}}\label{sec:ld_ingredient1}
	The first ingredient controls the boundary and non-basic-activity terms that arise when the BAR is applied to the exponential test function. Lemma~\ref{lem:ld_ingredient1} shows that, to first order, these terms reduce to the idleness relation in (\ref{eq:idle}).
	\begin{lemma}\label{lem:ld_ingredient1}
	Under the conditions of Proposition \ref{prop:ld}, for each $\theta\leq 0$ and $r\in (0,r_0)$,
		$$\begin{aligned}
			&\sum_{k=1}^{K}u_k\phi^{(r)}_{k}(\theta)+\sum_{k=1}^{K}\sum_{i\in \mathcal{I}(k)}d_{ik}\phi^{(r)}_{ik}(\theta)\\
			=&\sum_{k=1}^{K}u_k\Prob\left(\sum_{i\in \mathcal{I}(k)}Z^{(r)}_{ik}=0\right)+\sum_{k=1}^{K}\sum_{i\in \mathcal{I}(k)}d_{ik}\E\bigg[P_{ik}\Big({Z^{(r)}}\Big)\bigg]+O(r^{3/2}).		
		\end{aligned}
		$$
		Furthermore, (\ref{eq:idle}) implies
		$$\sum_{k=1}^{K}u_k\phi^{(r)}_{k}(\theta)+\sum_{k=1}^{K}\sum_{i\in \mathcal{I}(k)}d_{ik}\phi^{(r)}_{ik}(\theta)=r+O(r^{3/2}).$$
	\end{lemma}

	\begin{proof} To establish the first claim, it suffices to bound
		$$\begin{aligned}
			&\sum_{k=1}^{K}u_k\E\left[\left(1-f_{\theta}(Z^{(r)})\right)\mathbbm{1}\left(\sum_{i\in \mathcal{I}(k)}Z^{(r)}_{ik}=0\right)\right]
			+\sum_{k=1}^{K}\sum_{i\in \mathcal{I}(k)}d_{ik}\E\left[P_{ik}({Z^{(r)}})\left(1-f_{\theta}(Z^{(r)})\right)\right]\\
			\stackrel{(a)}{\leq}&r|\theta|\sum_{k=1}^{K}u_k\E\left[\left(\sum_{k'=1}^Ku_{k'}W_{k'}({Z}^{(r)})\right)\mathbbm{1}\left(\sum_{i\in \mathcal{I}(k)}Z^{(r)}_{ik}=0\right)\right]\\
			&+r|\theta|\sum_{k=1}^{K}\sum_{i\in \mathcal{I}(k)}d_{ik}\E\left[P_{ik}({Z^{(r)}})\left(\sum_{k'=1}^Ku_{k'}W_{k'}({Z}^{(r)})\right)\right]\\
			\stackrel{(b)}{\leq}& O(r^{3/2})
		\end{aligned}$$
		where $(a)$ follows from $1-e^{-x}\leq x$ for $x\geq0$ and the inequality implied by (\ref{eq:relation}), while $(b)$ follows from Lemma~\ref{lem:neg} for the first term and Lemma~\ref{lem:nbcon} for the second, noting that only activities with $d_{ik}>0$ contribute to the second term.
	\end{proof}
	
\subsection{Ingredient II for Proposition \ref{prop:ld}}\label{sec:ld_ingredient2}
The second ingredient handles the activity-specific service terms in the BAR. The following transform-level flow-balance identity allows us to express these terms directly in terms of the transform $\phi^{(r)}(\theta)$ and thereby close the limiting BAR.
\begin{lemma}[Asymptotic transform-level flow balance]\label{lem:transflowbalance}
    Under the conditions of Proposition \ref{prop:ld}, for each $i\in\mathcal{I}$ and each $\theta\leq 0$, 
\begin{equation}\label{eq:ssc2_1}\begin{aligned}\lim_{r\downarrow 0}\bigg[\lambda_{i}v_{i}\phi^{(r)}(\theta)-\sum_{k\in \mathcal{K}(i)}u_k\phi^{(r)}_{ik}(\theta)\bigg]=0.
    \end{aligned}\end{equation}
\end{lemma}
\begin{remark}
    Equivalently, Lemma~\ref{lem:transflowbalance} gives the transform-level flow balance for each job class:
    \begin{equation}\label{eq:ssc2_2} \begin{aligned}\lim_{r\downarrow 0}\bigg[\lambda_{i}&\phi^{(r)}(\theta)-\sum_{k\in \mathcal{K}(i)}\mu_{ik}\phi^{(r)}_{ik}(\theta)\bigg]=0.
    \end{aligned}\end{equation}
    The equivalence of (\ref{eq:ssc2_1}) and (\ref{eq:ssc2_2}) follows from (\ref{eq:dik}) and (\ref{eq:idle}), since, for each activity $(i,k)$,
    $$
    d_{ik}\phi^{(r)}_{ik}(\theta)\leq d_{ik}\E\left[P_{ik}\big({Z}^{(r)}\big)\right]=O(r).
    $$
    In Lemma \ref{lem:transflowbalance}, setting $\theta=0$ reduces the identity to the heavy-traffic limit of the usual rate-level flow balance equation for each class $i$, which matches the long-run average arrival rate with the corresponding average service-completion rate. More generally, for every $\theta\leq 0$, Lemma~\ref{lem:transflowbalance} shows that this balance continues to hold asymptotically when the arrival and activity-specific service flows are weighted by the exponential test function of the scaled queue length, providing a transform-level strengthening of the steady-state flow-balance relation. 
\end{remark}

\begin{proof}[Proof of Lemma \ref{lem:transflowbalance}]
    We first consider the case $\theta=0$. For $\theta<0$, we first establish the following identity for any $t>0$:
    \begin{equation}\label{eq:helper}	\begin{aligned} \lim_{r\downarrow 0}&\E\Bigg[\bigg(\lambda_{i}v_{i}
            -\sum_{k\in \mathcal{K}(i)}u_kP_{ik}({Z}^{(r)})\bigg)f_{\theta}(Z^{(r)})e^{r\theta\left(t\sum_{k\in \mathcal{K}(i)}v_{i}Z^{(r)}_{ik}\right)}\Bigg]=0,
    \end{aligned}\end{equation}
     where we refer to $t$ as a perturbation parameter. We then extend the identity to $t=0$ through an interchange-of-limits argument, which yields (\ref{eq:ssc2_1}). 
    \begin{enumerate}
        \item[(i)] When $\theta=0$,
        $(\ref{eq:ssc2_1})$ can be derived from  $(\ref{eq:flowbalance})$. Namely, we sum up $(\ref{eq:flowbalance})$ over $k\in \mathcal{K}(i)$ and hence obtain the class-level flow balance:
        $$
        \begin{aligned}
            \sum_{k\in \mathcal{K}(i)}\mu_{ik}\E\left[P_{ik}({Z^{(r)}})\right]=\lambda^{(r)}_i.
        \end{aligned}$$
        Taking the limit, multiplying both sides by $v_i$, and using $(\ref{eq:dik})$, we have
        $$\begin{aligned}\lambda_{i}v_{i}
            =\lim_{r\downarrow 0}\sum_{k\in \mathcal{K}(i)}(u_k-d_{ik})\E\left[P_{ik}({Z}^{(r)})\right]\stackrel{(a)}{=}\lim_{r\downarrow 0}\sum_{k\in \mathcal{K}(i)}u_k\E\left[P_{ik}({Z}^{(r)})\right], 
        \end{aligned}$$
        where $(a)$ follows from $(\ref{eq:nbcon_0})$.
        \item[(ii)] When $\theta<0$, we choose a test function of the form $\tilde{f}_{\theta}(z)=e^{r\theta\left(\sum_{i=1}^I\sum_{k\in \mathcal{K}(i)}c_iz_{ik}\right)}$ with constants $c_i\geq 0$ to be determined for $i\in\mathcal{I}$. The generator (\ref{eq:generator}) then becomes
        $$\begin{aligned}
            G\tilde{f}_{\theta}({z})=&\sum_{i=1}^{I}\lambda^{(r)}_i\left(e^{c_ir\theta}-1\right)\tilde{f}_{\theta}(z)
            +\sum_{k=1}^{K}\sum_{i\in \mathcal{I}(k)}\frac{u_k-d_{ik}}{v_i}P_{ik}({z})\left(e^{-c_ir\theta}-1\right)\tilde{f}_{\theta}(z).\\
        \end{aligned}$$
        Applying Lemma~\ref{lem:steady_exp} and using a second-order Taylor expansion together with $0\le\tilde f_\theta(z)\le1$, we obtain
        \begin{equation}\label{eq:equiv1}
            \begin{aligned}
                &\sum_{i=1}^{I}\lambda_ic_ir\theta \tilde{\phi}^{(r)}(\theta)
                +\sum_{k=1}^{K}\sum_{i\in \mathcal{I}(k)}\frac{u_k-d_{ik}}{v_i}(-c_ir\theta)\tilde{\phi}_{ik}^{(r)}(\theta)
                -\sum_{i=1}^{I}\lambda_ic_ir^2\theta \tilde{\phi}^{(r)}(\theta)\\
                +&\frac{1}{2}\sum_{i=1}^{I}\lambda_ic^2_ir^2\theta^2 \tilde{\phi}^{(r)}(\theta)
                +\frac{1}{2}\sum_{k=1}^{K}\sum_{i\in \mathcal{I}(k)}\frac{u_k-d_{ik}}{v_i}(c^2_ir^2\theta^2)\tilde{\phi}_{ik}^{(r)}(\theta)=O(r^3),
            \end{aligned}
        \end{equation}
        where we have used $\lambda_i^{(r)}=\lambda_i(1-r)$, and $\tilde\phi^{(r)}$ and $\tilde\phi_{ik}^{(r)}$ are defined as in $(\ref{eq:laplace})$ by replacing $f$ by $\tilde{f}$.
        Collecting the first-order terms in $r$, (\ref{eq:equiv1}) yields
        \begin{equation}\label{eq:equiv2}
            \begin{aligned}
                &\sum_{i=1}^{I}\lambda_ic_ir\theta \tilde{\phi}^{(r)}(\theta)
                +\sum_{k=1}^{K}\sum_{i\in \mathcal{I}(k)}\frac{u_k-d_{ik}}{v_i}(-c_ir\theta)\tilde{\phi}_{ik}^{(r)}(\theta)=O(r^2).\\
        \end{aligned}\end{equation}
        Now, for each fixed $i'\in\mathcal{I}$, we set $$\begin{aligned}
            &c_{i'}=v_{i'}(1+t),\quad t\geq 0,\\
            &c_i=v_i, \quad \text{ for all } i\in\mathcal{I}, i\neq i'.\end{aligned}$$
        Then $(\ref{eq:equiv2})$ becomes
        \begin{equation}\label{eq:equiv3} \begin{aligned}
            &\lambda_{i'}v_{i'} t\tilde{\phi}^{(r)}(\theta)+\sum_{i=1}^{I}\lambda_iv_i\tilde{\phi}^{(r)}(\theta)
            -\sum_{k\in\mathcal{K}(i')}(u_k-d_{i'k})t\tilde{\phi}_{i'k}^{(r)}(\theta)\\
            &-\sum_{k=1}^{K}u_k\tilde{\phi}^{(r)}(\theta)+\sum_{k=1}^{K}u_k\tilde{\phi}_{k}^{(r)}(\theta)+\sum_{k=1}^{K}\sum_{i\in \mathcal{I}(k)}d_{ik}\tilde{\phi}_{ik}^{(r)}(\theta)=O(r).\end{aligned}\end{equation}
By Lemma~\ref{lem:pd1}, since $\sum_{i=1}^{I}\lambda_iv_i=\sum_{k=1}^{K}u_k=1$, the corresponding terms involving $\tilde{\phi}^{(r)}(\theta)$ cancel. In addition, by Lemma~\ref{lem:idle},
   $$\begin{aligned}
            &\tilde{\phi}_{k}^{(r)}(\theta)\leq \Prob \bigg(\sum_{i\in \mathcal{I}(k)}Z^{(r)}_{ik}=0\bigg)=O(r),\\
            &\sum_{k=1}^{K}\sum_{i\in \mathcal{I}(k)}d_{ik}\tilde{\phi}_{ik}^{(r)}(\theta)
            \leq \sum_{k=1}^{K}\sum_{i\in \mathcal{I}(k)}d_{ik}\E\left[P_{ik}({Z}^{(r)})\right]=O(r).
        \end{aligned}$$
Using these bounds in $(\ref{eq:equiv3})$ and absorbing all terms involving $d_{ik}$ into the $O(r)$ term, we obtain 
$$\lambda_{i'}v_{i'} t\tilde{\phi}^{(r)}(\theta)-\sum_{k\in \mathcal{K}(i')}u_kt\tilde{\phi}_{i'k}^{(r)}(\theta)=O(r).$$
        Equivalently, for each $i'\in\mathcal I$,
        \begin{equation}\label{eq:equiv4}
            \begin{aligned}
                &\lambda_{i'}v_{i'} t\E\left[e^{r\theta\left(t\sum_{k\in \mathcal{K}(i')}v_{i'}Z^{(r)}_{i'k}\right)}f_{\theta}(Z^{(r)})\right]\\
                &-\sum_{k\in \mathcal{K}(i')}u_kt\E\left[P_{i'k}({Z}^{(r)})e^{r\theta\left(t\sum_{k\in \mathcal{K}(i')}v_{i'}Z^{(r)}_{i'k}\right)}f_{\theta}(Z^{(r)})\right]
                =O(r).\\
            \end{aligned}
        \end{equation}
        We distinguish two cases for $t\geq0$ in $(\ref{eq:equiv4})$: 
        \begin{enumerate}
            \item 
            If $t> 0$, we divide both sides of (\ref{eq:equiv4}) by $t$ and take $r\downarrow 0$. Then (\ref{eq:helper}) follows for each $t >0$ and each $ i\in\mathcal{I}$.
            \item
            To extend the result to $t=0$, we use the following lemma:
            \begin{lemma}\label{lem:interchg}
                For each $i\in\mathcal I$, if (\ref{eq:helper}) holds for every $t>0$, then it also holds for $t=0$.
            \end{lemma}
            The proof of Lemma~\ref{lem:interchg} is provided in Appendix~\ref{sec:interchg}, where we verify the conditions of the Moore--Osgood theorem \citep{Lawr1946} to justify interchanging the limits $t\downarrow0$ and $r\downarrow0$. This extends (\ref{eq:helper}) to $t=0$ and yields (\ref{eq:ssc2_1}).
        \end{enumerate}
    \end{enumerate}
\end{proof}

	\subsection{Proof of Proposition \ref{prop:ld}}\label{sec:ld}
With the two ingredients established above, the limiting BAR can now be closed: Lemma~\ref{lem:ld_ingredient1} controls the boundary and non-basic-activity terms, while Lemma~\ref{lem:transflowbalance} converts the remaining activity-specific service terms into the transform $\phi^{(r)}(\theta)$.
	\begin{proof}
 Applying Lemma \ref{lem:steady_exp} to $f_{\theta}(z)$ defined in $(\ref{eq:ef})$ and using a second-order Taylor expansion, with $0\leq f_{\theta}({z})\leq 1$, we obtain the following BAR after dividing both sides by $r\theta$: 
		$$
		\begin{aligned}
			&\sum_{k=1}^{K}u_k\phi_{k}^{(r)}(\theta)+\sum_{k=1}^{K}\sum_{i\in \mathcal{I}(k)}d_{ik}\phi_{ik}^{(r)}(\theta)
			-\sum_{i=1}^{I}\lambda_iv_ir \phi^{(r)}(\theta)\\
			&+\frac{1}{2}\sum_{i=1}^{I}\lambda_iv^2_ir\theta \phi^{(r)}(\theta)
			+\frac{1}{2}\sum_{k=1}^{K}\sum_{i\in \mathcal{I}(k)}u_kv_ir\theta\phi_{ik}^{(r)}(\theta)+O(r^2)=0.\\
		\end{aligned}
		$$
		Applying Lemma~\ref{lem:ld_ingredient1}, the BAR becomes
		$$
		\begin{aligned}
			&r+O(r^{3/2})-\sum_{i=1}^{I}\lambda_iv_ir \phi^{(r)}(\theta)+\frac{1}{2}\sum_{i=1}^{I}\lambda_iv^2_ir\theta \phi^{(r)}(\theta)
			+\frac{1}{2}\sum_{k=1}^{K}\sum_{i\in \mathcal{I}(k)}u_kv_ir\theta \phi_{ik}^{(r)}(\theta)
			=0.
		\end{aligned}
		$$
		Using $\sum_{i=1}^{I}\lambda_iv_i =1$ in Lemma \ref{lem:pd1}, and rearranging some terms, the BAR implies
		$$
		\begin{aligned}
			&r
			-r \phi^{(r)}(\theta)+\sum_{i=1}^{I}\lambda_iv^2_ir\theta \phi^{(r)}(\theta)+\frac{1}{2}\sum_{i=1}^{I}v_ir\theta \left(\sum_{k\in \mathcal{K}(i)}u_k\phi^{(r)}_{ik}(\theta)-\lambda_{i}v_{i}\phi^{(r)}(\theta)\right)=O(r^{3/2}).
		\end{aligned}
		$$
	Dividing both sides by $r$ and then applying Lemma~\ref{lem:transflowbalance}, the last term converges to zero as $r\downarrow0$. Hence,
		$$\begin{aligned}
			&(1-\theta\sum_{i=1}^{I}\lambda_iv^2_i)\lim_{r\downarrow 0} \phi^{(r)}(\theta)=1.\\
		\end{aligned}
		$$
		This completes the proof of Proposition~\ref{prop:ld}.
	\end{proof}
	
	\subsection{Proof of Theorem \ref{thm:main}}\label{sec:mdgwh_c}
	\begin{proof} 
		Denote 
		$$U^{(r)}\triangleq r \bigg(\sum_{k=1}^Ku_kW_k({Z}^{(r)})\bigg)=r\bigg(\sum_{k=1}^Ku^2_kT_k({Z}^{(r)})\bigg).$$
		Let $e=(1,\ldots,1)^\top$. Corollary~\ref{coro:ld_workload} can then be written in vector form as
		\begin{equation}{\label{19}}
			U^{(r)}e \stackrel{d}{\rightarrow} \tilde{X}e,\quad \text{as }r \downarrow 0.
		\end{equation}
		Now denote $Y^{(r)}=(Y_{1}^{(r)},\ldots,Y_{K}^{(r)})^\top$, where $$Y_{k'}^{(r)}\triangleq r\left(\sum_{k=1}^{K}u_k^2\right)\frac{1}{u_{k'}}W_{k'}({Z}^{(r)})=r\sum_{k=1}^{K}u_k^2T_{k'}({Z}^{(r)}),\quad k'\in\mathcal{K}.$$
		Taking the difference, we have
		$$\begin{aligned}
			&\Big\|U^{(r)}e-Y^{(r)}\Big\|\leq\sum_{k'=1}^K \left|U^{(r)}-Y_{k'}^{(r)}\right|
			\leq\sum_{k'=1}^K r\left[\sum_{k=1}^{K}u_k^2\left|T_{k}({Z}^{(r)})-T_{k'}({Z}^{(r)})\right|\right]\\
			\leq& \sum_{k'=1}^K r\sum_{k=1}^{K}u_k^2\left(T_{(K)}({Z}^{(r)})-T_{(1)}({Z}^{(r)})\right) \leq r K \left(T_{(K)}({Z}^{(r)})-T_{(1)}({Z}^{(r)})\right),\\
		\end{aligned}
		$$
		where the last inequality uses $\sum_{k=1}^{K}u_k^2\leq1$. By Proposition \ref{prop:sscmoment} and Markov's inequality,
		\begin{equation}{\label{20}}\begin{aligned}
				&\left\|U^{(r)}e-Y^{(r)}\right\|\stackrel{p}{\rightarrow} 0, \quad \text{as }r \downarrow 0.
			\end{aligned}
		\end{equation}
		Combining (\ref{19}) and (\ref{20}) and applying Theorem~3.1 of \cite{Bill1999}, we obtain
		$$Y^{(r)}\stackrel{d}{\rightarrow }\tilde{X}e, \quad \text{as }r \downarrow 0,$$
		i.e., 
		$$r\left(\sum_{k=1}^{K}u^2_k\right)\left(\frac{1}{u_{1}}W_{1}({Z}^{(r)}),\ldots,\frac{1}{u_{K}}W_{K}({Z}^{(r)})\right)^\top\stackrel{d}{\rightarrow }\tilde{X}e,$$
		where, recall, $\tilde{X}$ is an exponential random variable with mean $m=\sum_{i=1}^{I}\lambda_i v_i^2$.
		Then, by the scaling property of the exponential distribution, we have
		$$\begin{aligned}
			r \left(W_1({Z}^{(r)}),\ldots,W_K({Z}^{(r)})\right) \stackrel{d}{\rightarrow}\left(u_1,\ldots,u_K\right) X,\quad \text{as }r\downarrow0,
		\end{aligned}$$
		where $X$ is an exponential random variable with mean
$$\frac{\sum_{i=1}^{I}\lambda_i v_i^2}{\sum_{k=1}^{K}u_k^2}.$$
	\end{proof}
	
\section{Further Discussion}\label{sec:discussion}
\subsection{A Comparison of Routing and Scheduling Policies}\label{sec:xmodel}

Under the OQ architecture, we need to specify a routing policy and a scheduling policy, while under the IQ architecture, only a scheduling policy is needed. In this section, we use an X-model example to explore how routing and scheduling policies affect stability under different architectures. Figures~\ref{fig:x_arch1} and~\ref{fig:x_arch2} show the X model under the OQ and IQ architectures, respectively. We set the same parameters for both as
		$$\lambda_1 = 1.3\rho, \quad\lambda_2 = 0.4\rho, \quad\mu_1 =\mu_3=\mu_4 = 1, \quad\mu_2 = 0.5,\quad \rho=0.95. $$
	 Under this parameter setting, solving the associated static allocation problem shows that activity $(2,1)$ is non-basic, as indicated by the dashed edges in Figures~\ref{fig:x_arch1} and~\ref{fig:x_arch2}. For the OQ X-model, we compare JSQ and WWTA routing, both combined with the priority scheduling rule. For the IQ X-model, we compare the priority rule with MaxWeight scheduling. For priority scheduling, we use a preemptive shortest-mean-processing-first static buffer priority rule, with ties broken in favor of class~2; in this example, class~2 jobs therefore have preemptive priority over class~1 jobs whenever both are available. Each policy setting is simulated at 95\% system load (i.e., $\rho=0.95$). Figure~\ref{fig:xcomparison} shows the total number of jobs in the system over time, and Figure~\ref{fig:xcomparison_zoom} provides a zoomed-in comparison of the two stable policies, WWTA with priority scheduling under OQ and MaxWeight under IQ. 
    
    As shown in Figure~\ref{fig:xcomparison}, JSQ with priority scheduling under the OQ architecture is unstable, while WWTA with the same priority scheduling is stable. This shows that using queue lengths alone for routing may not be sufficient in systems with heterogeneous job classes and server types: even when the overall system load is below capacity, JSQ does not distinguish between slow and fast activities and may route too many jobs through slow activities. Under the IQ architecture, priority scheduling is unstable, whereas MaxWeight scheduling is stable. This indicates that, in IQ systems, the choice of scheduling policy can be critical for stability, as an intuitive priority rule may not suffice. 
    
    Taken together, these comparisons illustrate how the load-balancing role is handled differently across the two architectures: under IQ, jobs wait in class queues and load balancing is maintained through scheduling decisions, whereas under OQ with immediate job assignment, routing decisions play the critical role in balancing the load, leaving the choice of scheduling policy as an additional degree of freedom. In particular, WWTA with the priority scheduling rule is stable in the OQ X-model, whereas the same scheduling rule is unstable in the IQ X-model.
		
		\begin{figure}
			\begin{minipage}[t]{0.5\linewidth}
				\centering
				\begin{tikzpicture}[inner sep=1.5mm]
					\node[server,minimum size=0.45in] at (6,4) (S1)  {$S_1$};
					\node[buffer] (B1) [above=0.25in of S1.north west] {$B_{1}$};
					\node[server,minimum size=0.45in] (S2) [right=0.8in of S1] {$S_2$};
					\node[vbuffer] (B2) [above=0.25in of S1.north east] {$B_{4}$};
					\node[buffer] (B3) [above=0.25in of S2.north west] {$B_{2}$};
					\node[buffer] (B4) [above=0.25in of S2.north east] {$B_{3}$};
					\node[routing] (R1) [above=0.4in of B1.north] {$R_1$};
					\node[routing] (R2) [above=0.4in of B4.north] {$R_2$};
					\coordinate (source1) at ($ (B1.north) + (0in, 0.8in)$) {};
					\coordinate (source2) at ($ (B4.north) + (0in, 0.8in)$) {};
					\coordinate (departure1) at ($ (S1.south) + (0in, -.2in)$) {};
					\coordinate (departure2) at ($ (S2.south) + (0in, -.2in)$) {};

					\draw[thick,->] (source1) -| (R1);
					\draw[thick,->] (source2) -| (R2);
					\draw[thick,->] (R1.south) -| (B1);
					\draw[thick,->] (R1.south) -- ++(B3);
					\draw[thick,->, dashed] (R2.south) -- ++(B2);
					\draw[thick,->] (R2.south) -| (B4);	
					\draw[thick,->] (S1.south) -| (departure1);
					\draw[thick,->] (S2.south) -| (departure2);
					\draw[thick,->] (B1.south) -| (S1.north west);
					\draw[thick,->] (B2.south) -| (S1.north east);
					\draw[thick,->] (B3.south) -| (S2.north west);
					\draw[thick,->] (B4.south) -| (S2.north east);
					\path[thick]
					;
					\node [below ,align=left] at ($(departure1)+(0in,0)$) {departure 1};
					\node [below,align=left] at ($(departure2)+(0in,0)$) {departure 2};
					\node [above,align=left] at (source1.south) {class 1\\ arrivals};
					\node [above,align=left] at (source2.south) {class 2\\ arrivals};
					
					\node [left,align=left] at ($(source1)+(0in,-.1in)$) {\textbf{$\lambda_1$}};
					\node [left,align=left] at ($(source2)+(0in,-.1in)$) {\textbf{$\lambda_2$}};
					
					\node [left,align=left] at ($(B1)+(0in,-.3in)$) {\textbf{$\mu_1$}};
					\node [left,align=left] at ($(B2)+(0in,-.3in)$) {\textbf{$\mu_4$}};
					\node [right,align=left] at ($(B3)+(0in,-.3in)$) {\textbf{$\mu_2$}};	
					\node [right,align=left] at ($(B4)+(0in,-.3in)$) {\textbf{$\mu_3$}};	
				\end{tikzpicture}
				\caption{X model under OQ architecture.}
				\label{fig:x_arch1}
			\end{minipage}
			\begin{minipage}[t]{0.5\linewidth}
				\centering
				\begin{tikzpicture}[inner sep=1.5mm]
					\node[server,minimum size=0.45in] at (6,4) (S1)  {$S_1$};
					\node[buffer] (B1) [above=0.6in of S1.north] {$B_{1}$};
					\node[server,minimum size=0.45in] (S2) [right=0.8in of S1] {$S_2$};
					\node[buffer] (B2) [above=0.6in of S2.north] {$B_{2}$};
					\coordinate (source1) at ($ (B1.north) + (0in, .4in)$) {};
					\coordinate (source2) at ($ (B2.north) + (0in, .4in)$) {};
					\coordinate (departure1) at ($ (S1.south) + (0in, -.2in)$) {};
					\coordinate (departure2) at ($ (S2.south) + (0in, -.2in)$) {};
					
					\draw[thick,->] (source1) -| (B1);
					\draw[thick,->] (source2) -| (B2);
					
					\draw[thick,->] (S1.south) -| (departure1);
					\draw[thick,->] (S2.south) -| (departure2);
					\draw[thick,->] (B1.south) -- ++ (S1);
					\draw[thick,->] (B1.south) -- ++ (S2);
					\draw[thick,->] (B2.south) -- ++ (S2);
					\draw[thick,->, dashed] (B2.south) -- ++(S1);
					\path[thick]
					;
					
					\node [below ,align=left] at ($(departure1)+(0in,0)$) {departure 1};
					\node [below,align=left] at ($(departure2)+(0in,0)$) {departure 2};
					\node [above,align=left] at (source1.south) {class 1\\ arrivals};
					\node [above,align=left] at (source2.south) {class 2\\ arrivals};
					
					\node [left,align=left] at ($(source1)+(0.02in,-.2in)$) {\textbf{$\lambda_1$}};
					\node [left,align=left] at ($(source2)+(0.02in,-.2in)$) {\textbf{$\lambda_2$}};
					
					\node [left,align=left] at ($(B1)+(-.0in,-.6in)$) {\textbf{$\mu_{1}$}};
					\node [left,align=left] at ($(B2)+(-.2in,-0.7in)$) {\textbf{$\mu_{2}$}};
					\node [left,align=left] at ($(B2)+(.45in,-.6in)$) {\textbf{$\mu_{3}$}};
					\node [left,align=left] at ($(S1)+(0.5in,.1in)$) {\textbf{$\mu_4$}};
				\end{tikzpicture}
				\caption{X model under IQ architecture.}
				\label{fig:x_arch2}
			\end{minipage}
			
		\end{figure}	
		
\begin{figure}[H]
	\centering
	\begin{subfigure}[t]{0.49\textwidth}
		\centering
		\includegraphics[width=\linewidth]{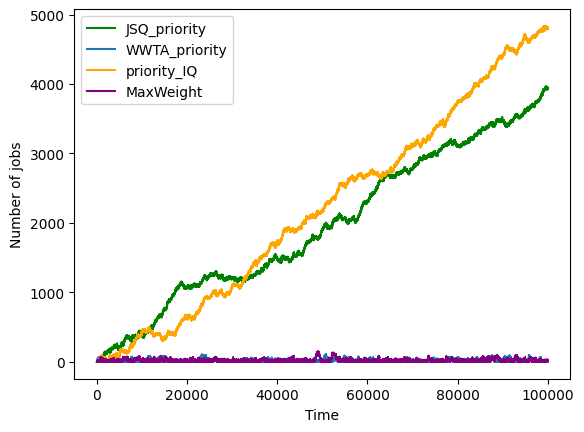}
		\caption{Comparison of all four policies.}
		\label{fig:xcomparison}
	\end{subfigure}
	\hfill
	\begin{subfigure}[t]{0.49\textwidth}
		\centering
		\includegraphics[width=\linewidth]{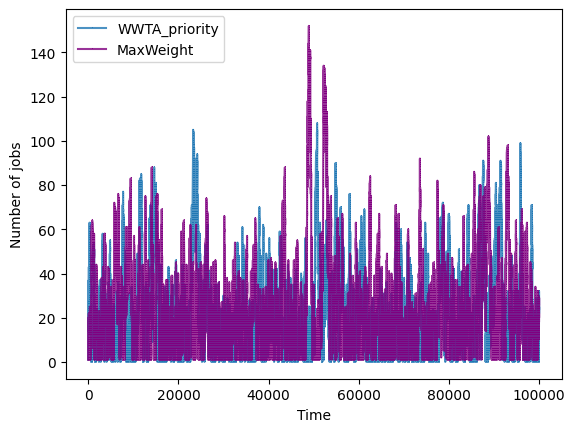}
		\caption{Zoom-in comparison of two stable policies.}
		\label{fig:xcomparison_zoom}
	\end{subfigure}
	\caption{Stability comparison under 95\% load.}
	\label{fig:xcomparison_all}
\end{figure}

   \subsection{Service Policies and Scheduling Flexibility}\label{sec:service_policy}
In OQ systems, a scheduling policy determines how each server allocates its service effort to different job classes. In this section, we introduce several scheduling rules that can be used as local service policies while WWTA handles load balancing through routing. We use a W-model example to demonstrate that, although the choice of scheduling policy does not affect the theoretical results established in the paper, it may affect other performance measures, such as delay. 

We consider two types of service policies that can be implemented with WWTA. The first type of service policy is the head-of-line proportional
	processor sharing (HLPPS) scheduling policy \citep{Bram1998}.
	Under HLPPS, all nonempty buffers receive
	service simultaneously. 
	For each server $k\in\mathcal{K}$, 
	the proportion
	of effort that server $k$ allocates among classes $\mathcal{I}(k)$ at any
	time is given by 
\begin{align}
		\label{eq:hlpps}
		P_{ik}(z)=\frac{z_{ik}}{\sum_{i'\in \mathcal{I}(k)}z_{i'k}}, \quad i\in \mathcal{I}(k), \end{align}
	when the state is $z=(z_{ik})$. Here, we
	adopt the convention that ${0}/{0}=0$, which means that when
	$\sum_{i\in \mathcal{I}(k)}z_{ik}=0$, server $k$ is idle.  Therefore, the
	instantaneous service rate for activity $j=(i,k)$ is
	$\mu_{ik}P_{ik}(z)$.  The
	implementation of the HLPPS policy does not require
	knowledge of system parameters, such as arrival rates or
	service rates. It also does not depend on the routing policy, but only on the relative queue sizes of the buffers at each server.
	
	The allocation in (\ref{eq:hlpps}) can also be generalized to 
	\begin{equation}\label{eq:ghlpps}
		P_{ik}({z})=\frac{c_{ik}z_{ik}}{\sum_{i'\in \mathcal{I}(k)}c_{i'k}z_{i'k}}, \quad  i\in\mathcal{I}(k), k\in\mathcal{K},
	\end{equation}
	where $c=(c_{ik})>0$ is a given vector of positive numbers.
	We call the service policy using allocation (\ref{eq:ghlpps})
	a  \emph{generalized HLPPS} policy with weights $c=(c_{ik})$. 
	
	The second type of service policy is the static buffer
	priority (SBP) scheduling policy. Each SBP policy specifies a fixed ranking of the buffers at each server. Given a ranking, we write
	$(i',k)\prec (i,k)$ to denote that buffer $(i',k)$ has preemptive priority over buffer $(i,k)$. Formally, we can define the SBP service policy for server $k$ in terms of $P_{ik}(z)$ values as follows
	$$P_{ik}({z})=\mathbbm{1}\Big(\sum_{(i',k)\prec (i,k)}z_{i'k}=0,z_{ik}>0 \Big).$$
	
	\begin{figure}[h]
		\centering
		\begin{tikzpicture}[inner sep=1.5mm]
			\node[server,minimum size=0.45in] at (6,4) (S1)  {$S_1$};
			\node[buffer] (B1) [above=0.4in of S1.north west] {$B_{11}$};
			\node[server,minimum size=0.45in] (S2) [right=0.9in of S1] {$S_2$};
			\node[buffer] (B2) [above=0.4in of S1.north east] {$B_{21}$};
			\node[buffer] (B3) [above=0.4in of S2.north west] {$B_{22}$};
			\node[buffer] (B4) [above=0.4in of S2.north east] {$B_{32}$};
			\node[vbuffer] (B6) [left=0.05in of B1] {$B_{31}$};
			\node[vbuffer] (B5) [right=0.05in of B4]{$B_{12}$};
			\node[center](C1)[left=0.2in of B1]{};
			\node[center](C2)[right=0.35in of B2]{};
			\node[center](C3)[right=0.2in of B4]{};
			\node[routing] (R1) [above=0.7in of C1.north east] {$R_1$};
			\node[routing] (R2) [above=0.7in of C2.north west] {$R_2$};
			\node[routing] (R3) [above=0.7in of C3.north west] {$R_3$};
			\coordinate (source1) at ($ (C1.north east) + (0in, 1.2in)$) {};
			\coordinate (source2) at ($ (C2.north west) + (0in, 1.2in)$) {};
			\coordinate (source3) at ($ (C3.north west) + (0in, 1.2in)$) {};
			\coordinate (departure1) at ($ (S1.south) + (0in, -.2in)$) {};
			\coordinate (departure2) at ($ (S2.south) + (0in, -.2in)$) {};

			\draw[thick,->] (source1) -| (R1);
			\draw[thick,->] (source2) -| (R2);
			\draw[thick,->] (source3) -| (R3);
			\draw[thick,->] (R1.south) -- ++ (B1);
			\draw[thick,->] (R2.south) -- ++(B2);
			\draw[thick,->, dashed] (R1.south) -- ++ (B5);
			\draw[thick,->, dashed] (R3.south) -- ++ (B6);
			\draw[thick,->] (R2.south) -- ++ (B3);	
			\draw[thick,->] (R3.south) -- ++ (B4);	
			\draw[thick,->] (S1.south) -| (departure1);
			\draw[thick,->] (S2.south) -| (departure2);
			
			\draw[thick,->] (B1.south) -| (S1.north west);
			\draw[thick,->] (B2.south) -| (S1.north east);
			\draw[thick,->] (B3.south) -| (S2.north west);
			\draw[thick,->] (B4.south) -| (S2.north east);
			\draw[thick,->] (B6.south) -- ++ (S1);	
			\draw[thick,->] (B5.south) -- ++ (S2);		
			
			\path[thick]
			;
			
			\node [below ,align=left] at ($(departure1)+(0in,0)$) {Server 1\\ departures};
			\node [below,align=left] at ($(departure2)+(0in,0)$) {Server 2\\ departures};
			
			\node [above,align=left] at (source1.south) {class 1\\ arrivals};
			\node [above,align=left] at (source2.south) {class 2\\ arrivals};
			\node [above,align=left] at (source3.south) {class 3\\ arrivals};
			
			\node [left,align=left] at ($(source1)+(0in,-.1in)$) {\textbf{$\lambda_1$}};
			\node [right,align=left] at ($(source2)+(0in,-.1in)$) {\textbf{$\lambda_2$}};
			\node [right,align=left] at ($(source3)+(0in,-.1in)$) {\textbf{$\lambda_3$}};	
			
			\node [left,align=left] at ($(B1)+(0.35in,-.4in)$) {\textbf{$\mu_{11}$}};
			\node [right,align=left] at ($(B2)+(0in,-.4in)$) {\textbf{$\mu_{21}$}};
			\node [left,align=left] at ($(B3)+(0in,-.4in)$) {\textbf{$\mu_{22}$}};
			\node [right,align=left] at ($(B4)+(-0.3in,-.4in)$) {\textbf{$\mu_{32}$}};
			\node [left,align=left] at ($(B1)+(-0.2in,-.4in)$) {\textbf{$\mu_{31}$}};
			\node [right,align=left] at ($(B4)+(0.2in,-.4in)$) {\textbf{$\mu_{12}$}};
		\end{tikzpicture}
		\caption{W model under the OQ architecture.}
		\label{fig:wmodel}
	\end{figure}
	
We next examine how the choice of scheduling policy affects delay performance. We use a W model with three job classes and two servers,  where the name ``W" comes from its W-shaped structure. The W model provides a richer setting for illustrating the impact of local scheduling decisions than the X model considered above. Figure~\ref{fig:wmodel} depicts the W model under the OQ architecture. We set
$$\begin{aligned}&\lambda_1 = 4\rho, \quad\lambda_2 = 1.3\rho, \quad\lambda_3 = 0.4\rho, \quad\mu_{11} =8,\quad\mu_{21} =2,\quad \mu_{31} =0.25,\\&\mu_{12} =0.25,\quad \mu_{22} =0.5, \quad\mu_{32} =1,\quad \rho\in(0,1). \end{aligned}$$
We implement WWTA with each of the two service policies introduced above: HLPPS and the shortest-mean-processing-first SBP rule. As a reference, we also implement MaxWeight scheduling under the corresponding IQ architecture, where jobs wait in class-specific queues, as in the IQ X-model shown in Figure~\ref{fig:x_arch2}. Under these system parameters, activities $(3,1)$ and $(1,2)$ are non-basic. Our performance metric is the average of the class-specific mean completion times, where the completion time of a job is measured from its arrival to its service completion and departure.
	
	We conduct the simulations at eight system load levels $\rho$ ranging from $92\%$ to $99\%$.  
	For each combination of load level and policy, we perform 30 independent replications of a discrete-event simulation, each run for 50,000 time units starting from an empty system, to obtain $95\%$ confidence intervals for this performance metric. As shown in Figure~\ref{fig:comparisonW}, the two WWTA implementations exhibit different delay performance, with SBP consistently yielding shorter average completion times than HLPPS. WWTA with either HLPPS or SBP scheduling yields shorter average completion times than MaxWeight over the range of loads considered. The performance differences become more pronounced as the system approaches heavy traffic.
    
	\begin{figure}[h]
		\centering
		\includegraphics[width=0.8\linewidth]{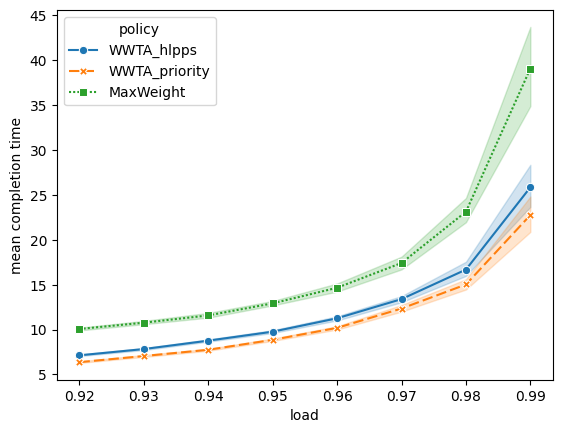}
		\caption{Average of the class-specific mean completion times in the W model.}
		\label{fig:comparisonW}
	\end{figure}

Taken together, these results illustrate that the scheduling flexibility enabled by WWTA can have a substantial impact on delay performance, with local service policies potentially tailored even on a server-by-server basis. This suggests an interesting direction for future work: optimizing local service policies for delay, cost, or value objectives, while leaving system-level load balancing and the handling of non-basic activities to WWTA routing.

	\setlength{\bibsep}{0.85pt}{
		\bibliography{xu20260906}
	}
	
	\appendix
	\begin{appendix}
		\addtocontents{toc}{\protect\setcounter{tocdepth}{1}}
		\makeatletter
		\addtocontents{toc}{%
			\begingroup
			\let\protect\l@chapter\protect\l@section
			\let\protect\l@section\protect\l@subsection
		}
		\makeatother

		\section{Proofs in Section \ref{sec:main} and Section \ref{sec:outline1}}
		\subsection{Proof of Lemma \ref{lem:pd1}}\label{sec:pd1_p} 
		
		\begin{proof}
			Under Assumption~\ref{assu1}, strong duality implies that the dual LP has an optimal solution $({v}^*,{u}^*)$ with zero duality gap. Therefore,
			$$\sum_{i=1}^{I}\lambda_iv^*_i=\rho^*=1.$$
			Since $\rho^*=1$,  complementary slackness for the primal variable $\rho$ gives
			$$\sum_{k=1}^{K}u^*_k =1.$$
			Moreover, $u^*_k\geq0$ follows from dual feasibility. For each activity $j=(i,k)$, complementary slackness also gives
            $$
            x_j^*\big[(u^*A)_j-(v^*R)_j\big]=0.
            $$
            Hence, if $x_{ik}^*>0$, then $(v^*R)_j=(u^*A)_j$. Since $(v^*R)_j=\mu_{ik}v_i^*$ and $(u^*A)_j=u_k^*$, we obtain
            $$
            \mu_{ik}v_i^*=u_k^*.
            $$
		\end{proof}	
		
		\subsection{Proof of Lemma \ref{lem:pd2}}\label{sec:pd2_p}
		\begin{proof} 
			Starting with any server $k$, denote $u_k^*=a\geq0$. By Assumption~\ref{assu2}, server $k$ is connected to all other servers through chains of basic activities. As an illustration, suppose that server $k$ has two communicating servers $k_1$ and $k_2$, which communicate directly with server $k$ through classes $i_1$ and $i_2$, respectively:
			$$\begin{aligned}
		 \text{server }k~~&\stackrel{\text{class }i_1}{\longrightarrow}& \text{server }k_1\\
				\downarrow _{\text{class } i_2}&&\\
				\text{server }k_2&&\\
			\end{aligned}$$
			In other words, server $k$ communicates with $k_1$ through the basic activities $(i_1,k)$ and $(i_1,k_1)$, and with $k_2$ through the basic activities $(i_2,k)$ and $(i_2,k_2)$.
			
			For any optimal primal solution $x^*$ and any optimal dual solution $(v^*,u^*)$, complementary slackness gives $\mu_{ik}v_i^*=u_k^*$ whenever $x_{ik}^*>0$. Hence, this equality holds along all basic activities used to establish communication, even if they correspond to different optimal primal solutions. Thus, we can write
			$$v^*_{i_1}=\frac{u^*_{k}}{\mu_{i_1{k}}}=\frac{a}{\mu_{i_1{k}}},\quad v^*_{i_2}=\frac{u^*_{k}}{\mu_{i_2{k}}}=\frac{a}{\mu_{i_2{k}}}; $$ 
			$$\begin{aligned}
				&u^*_{k_1} = v^*_{i_1}\mu_{i_1k_1} = a\frac{\mu_{i_1k_1}}{\mu_{i_1{k}}},\quad u^*_{k_2} = v^*_{i_2}\mu_{i_2k_2}=a\frac{\mu_{i_2k_2}}{\mu_{i_2{k}}}.
			\end{aligned}$$
			Thus, $u_{k_1}^*$ and $u_{k_2}^*$ can be expressed as $a$ multiplied by ratios of service rates.
			
			Similarly, starting with servers $k_1$ and $k_2$, we can find other communicating servers through basic activities and obtain $u_{k_3}^*,u_{k_4}^*,\ldots$, each expressed as $a$ multiplied by ratios of service rates. Continuing this procedure, Assumption~\ref{assu2} guarantees that we can eventually reach all servers and obtain such an expression for each server. Finally, by Lemma~\ref{lem:pd1}(ii), $\sum_{k=1}^K u_k^*=1$, which uniquely determines $a$. Tracing back through these equations then uniquely determines each component of $(v^*,u^*)$.
			
			In the discussion above, at each step we select one class through which two servers communicate. This choice need not be unique, and some equalities corresponding to other basic activities may therefore not be used in determining $(v^*,u^*)$. However, the choice of communication paths does not affect uniqueness. Once the selected communication paths and the normalization $\sum_k u_k^*=1$ uniquely determine $(v^*,u^*)$, the unused equalities impose only additional constraints that this unique optimal dual solution must satisfy. 
			
			For Lemma~\ref{lem:pd2}(ii), since $u_k^*\geq0$ for all $k$ and $\sum_{k=1}^K u_k^*=1$, there exists at least one server $k$ with $u_k^*>0$. Starting from this server and applying the same equations as above, we obtain $u_k^*>0$ for all $k\in\mathcal{K}$. Since each class has at least one basic activity and $\mu_{ik}>0$, the relation $\mu_{ik}v_i^*=u_k^*$ then implies $v_i^*>0$ for all $i\in\mathcal{I}$.
		\end{proof}	
		
		\subsection{Proof of Lemma \ref{lem:steady_exp}}\label{sec:steady} \begin{proof}    Assume $f: \mathbb{Z}^J_+ \rightarrow \mathbb{R}$ satisfies $|f({z})| \leq C\big(1+\sum_{k=1}^{K}W^{n}_k(z)\big)$ for some $C>0$. Following \cite{BravDaiFeng2016}, we use the following sufficient condition given in \cite{Hend1997,GlynZeev2008} for $\E\left[Gf\left({Z}^{(r)}\right)\right]=0$:
			$$
	\mathbb{E}\left[\left|G\left({Z}^{(r)}, {Z}^{(r)}\right) f\big({Z}^{(r)}\big)\right|\right]<\infty,
			$$
			where $G({z},{z})$ is the diagonal element of the generator matrix $G$ corresponding to state ${z}$. First, we have
			$$\begin{aligned}
				&\left|G\left({Z}^{(r)}, {Z}^{(r)}\right)\right| \\
				=& \bigg|-\bigg(\sum_{i=1}^{I}\lambda_i^{(r)}\sum_{k\in \mathcal{K}{(i)}}\mathbbm{1}\left(k=L^{(i)}({Z}^{(r)})\right)+\sum_{k=1}^{K}\sum_{i\in \mathcal{I}(k)}\mu_{ik}P_{ik}\left({Z}^{(r)}\right)\bigg)\bigg|\\
				\leq&\sum_{i=1}^{I}\lambda_i+ \sum_{k=1}^{K}\sum_{i\in \mathcal{I}(k)}\mu_{ik}.
			\end{aligned}$$
			Therefore, by the assumed bound on $f$, it suffices to show that $\sum_{k=1}^{K}\mathbb{E}\Big[W_k^{n}\Big(Z^{(r)}\Big)\Big]<\infty$. Define $$V({z}) =\frac{1}{n+1} \sum_{k=1}^{K}\frac{1}{u_k^{n-1}}W^{n+1}_k(z).$$ 
			The generator of $V$ can be written in terms of the binomial expansion as follows:
			$$\begin{aligned}
				GV({z})=
				&\frac{1}{n+1}\sum_{i=1}^{I}\lambda_i^{(r)}\sum_{k\in \mathcal{K}(i)}\frac{1}{u_k^{n-1}}\left(\sum_{\ell=1}^{n+1}{n+1\choose\ell}m_{ik}^\ell W_k^{n+1-\ell}(z)\right)\mathbbm{1}\left(k=L^{(i)}(z)\right)\\
				&+\frac{1}{n+1}\sum_{k=1}^{K}\sum_{i\in \mathcal{I}(k)}\frac{\mu_{ik}P_{ik}(z)}{u_k^{n-1}}\left(\sum_{\ell=1}^{n+1}{n+1\choose\ell}\left(-m_{ik}\right)^\ell W_k^{n+1-\ell}(z)\right)\\
				=&\sum_{i=1}^{I}\lambda_i^{(r)}\sum_{k\in \mathcal{K}(i)}\frac{m_{ik}}{u_k^{n-1}}W^n_k(z)\mathbbm{1}\left(k=L^{(i)}(z)\right)\\
				&-\sum_{k=1}^{K}\frac{1}{u_k^{n-1}}W^n_k(z)\sum_{i\in \mathcal{I}(k)}P_{ik}(z) +o\Big(\sum_{k=1}^KW^n_k(z)\Big)\\
				\stackrel{(a)}{\leq}& \sum_{i=1}^{I}\lambda_i^{(r)}\sum_{k\in \mathcal{K}(i)}\frac{1}{v_i^{n-1}}\left(m_{ik}W_k(z)\right)^n\mathbbm{1}\left(k=L^{(i)}(z)\right)-\sum_{k=1}^{K}\frac{1}{u_k^{n-1}}W^n_k(z) +o\Big(\sum_{k=1}^KW^n_k(z)\Big)\\
				\stackrel{(b)}{=}&\sum_{i=1}^{I}\frac{\lambda_i^{(r)}}{v_i^{n-1}}\min_{k\in \mathcal{K}(i)}\left(m_{ik}W_k(z)\right)^n-\sum_{k=1}^{K}\frac{1}{u_k^{n-1}}W^n_k(z) +o\Big(\sum_{k=1}^K W^n_k(z)\Big),\\
			\end{aligned}$$
			where the terms of lower order are written as
$o\big(\sum_{k=1}^K W_k^n(z)\big)$ as $\sum_{k=1}^K W_k^n(z)\to\infty$. $(a)$ follows from (\ref{eq:dik}) and the non-idling property of $\mathcal P$, and $(b)$ follows from the definition of the WWTA policy. Under Assumption~\ref{assu1}, we choose an optimal processing plan $x^*$ and define $\lambda_{ik}\triangleq  x^*_{ik}\mu_{ik}$ for $(i,k)\in\mathcal{J}$. Then $\lambda_{ik}$ satisfies $\lambda_{i}=\sum_{k\in \mathcal{K}(i)}\lambda_{ik}$ and $\sum_{i\in \mathcal{I}(k)}\lambda_{ik}m_{ik} =1$. Let $\lambda^{(r)}_{ik} = \lambda_{ik}(1-r)$. Then $\lambda_{i}^{(r)}=\sum_{k\in \mathcal{K}(i)}\lambda_{ik}^{(r)}$ and $\sum_{i\in \mathcal{I}(k)}\lambda_{ik}^{(r)}m_{ik}=1-r <1$. Following an argument similar to that in the proof of Theorem 11.6 of \cite{DaiHarr2020}, we have
			$$\begin{aligned}
				GV({z})
				\leq &\sum_{i=1}^{I}\frac{1}{v_i^{n-1}}\sum_{k\in \mathcal{K}(i)}\lambda_{ik}^{(r)}\left(\frac{1}{\mu_{ik}}W_k(z)\right)^n-\sum_{k=1}^{K}\frac{1}{u_k^{n-1}}W^n_k(z) +o\Big(\sum_{k=1}^KW^n_k(z)\Big)\\
				\stackrel{(c)}{=}&\sum_{i=1}^{I}\sum_{k\in \mathcal{K}(i)}^b\frac{\lambda_{ik}^{(r)}}{\mu_{ik}}\frac{1}{u_k^{n-1}}W^n_k(z)+\sum_{i=1}^{I}\sum_{k\in \mathcal{K}(i)}^{nb}\frac{\lambda_{ik}^{(r)}}{\mu_{ik}}\frac{1}{(u_k-d_{ik})^{n-1}}W^n_k(z)\\
				&-\sum_{k=1}^{K}\frac{1}{u_k^{n-1}}W^n_k(z) +o\Big(\sum_{k=1}^KW^n_k(z)\Big)\\
				=&\sum_{i=1}^{I}\sum_{k\in \mathcal{K}(i)}\frac{\lambda_{ik}^{(r)}}{\mu_{ik}}\frac{1}{u_k^{n-1}}W^n_k(z)-\sum_{k=1}^{K}\frac{1}{u_k^{n-1}}W^n_k(z)+o\Big(\sum_{k=1}^KW^n_k(z)\Big)\\
				&+\sum_{i=1}^{I}\sum_{k\in \mathcal{K}(i)}^{nb}\frac{\lambda_{ik}^{(r)}}{\mu_{ik}}\frac{1}{(u_k-d_{ik})^{n-1}}W^n_k(z)-\sum_{i=1}^{I}\sum_{k\in \mathcal{K}(i)}^{nb}\frac{\lambda_{ik}^{(r)}}{\mu_{ik}}\frac{1}{u_k^{n-1}}W^n_k(z)\\
				=&\sum_{k=1}^{K}\frac{1}{u_k^{n-1}}\left(\sum_{i\in \mathcal{I}(k)}\frac{\lambda_{ik}^{(r)}}{\mu_{ik}}-1\right)W^n_k(z)+o\Big(\sum_{k=1}^KW^n_k(z)\Big)\\
				&+\sum_{i=1}^{I}\sum_{k\in \mathcal{K}(i)}^{nb}\frac{\lambda_{ik}^{(r)}}{\mu_{ik}}\left(\frac{1}{(u_k-d_{ik})^{n-1}}-\frac{1}{u_k^{n-1}}\right)W^n_k(z),\\
			\end{aligned}$$
			where $(c)$ follows by considering basic and non-basic activities separately and applying (\ref{eq:dik}). The last term is zero because, for each activity with $d_{ik}>0$, complementary slackness gives $x_{ik}^*=0$, and hence $\lambda_{ik}^{(r)}=0$. Therefore, the generator becomes
			$$\begin{aligned}
				GV({z})=&\sum_{k=1}^{K}\frac{1}{u_k^{n-1}}\left(\sum_{i\in \mathcal{I}(k)}\frac{\lambda_{ik}^{(r)}}{\mu_{ik}}-1\right)W^n_k(z)+o\Big(\sum_{k=1}^KW^n_k(z)\Big)\\
				\leq& -\frac{r}{\bar{u}^{n-1}} \sum_{k=1}^{K}W^n_k(z)+o\Big(\sum_{k=1}^KW^n_k(z)\Big),
			\end{aligned}$$
			where  $\bar{u}\triangleq\max_{k\in\mathcal{K}}u_k>0$. Therefore, there exist constants $c > 0$ and $c_w>0$ such that, whenever $\sum_{k=1}^{K}W^n_k(z)\geq c_w$,
			$$-\frac{r}{\bar{u}^{n-1}} \sum_{k=1}^{K}W^n_k(z)+o\Big(\sum_{k=1}^KW^n_k(z)\Big)\leq -cr\sum_{k=1}^{K}W^n_k(z).$$
			Consequently, there exists $d>0$ such that
			$$GV({z})\leq -cr\sum_{k=1}^{K}W^n_k(z) + d\mathbbm{1}\bigg(\sum_{k=1}^{K}W^n_k(z)< c_w\bigg).$$
			By Theorem~4.3 of \cite{MeynTwee1993b}, we have 
			$$\sum_{k=1}^{K}\mathbb{E}\bigg[W_k^{n}\Big(Z^{(r)}\Big)\bigg]<\infty.$$
		\end{proof}

		\section{Proofs in Section \ref{sec:ssc_moment}}
		\subsection{Proof of Lemma \ref{lem:ssc1}}\label{sec:ssc_lm1}
		Before proving Lemma \ref{lem:ssc1}, we first present the following lemma.
		\begin{lemma}\label{alge_inequality}
			For any $p,x\in\mathbb{R}$ with $p>0$ and $x\geq -1$, the following inequality holds:
			$$\frac{x}{(1+x)^p +1}\leq\frac{x}{2}.$$
		\end{lemma}
		\begin{proof}[Proof of Lemma \ref{alge_inequality}]
			If $x=0$, equality holds. If $x>0$, then $(1+x)^p+1>2$. If $-1\leq x<0$, then $(1+x)^p+1<2$. Therefore, the inequality holds for all $x\geq-1$.
		\end{proof}
		\begin{proof}[Proof of Lemma~\ref{lem:ssc1}]
        When a class $i$ job is routed to server $k$, the incremental change in the perpendicular component $W_{\perp}(z)$ is
	\begin{equation}\label{eq:increment}\begin{aligned}\delta_{ik} &= W(z) + m_{ik}e^{(k)}-\big[W_{\parallel}(z) + \langle m_{ik}e^{(k)}, \hat{u}\rangle \hat{u}\big] - [W(z)-W_{\parallel}(z) ]\\& =  m_{ik}e^{(k)}-\langle m_{ik}e^{(k)}, \hat{u}\rangle \hat{u}.\end{aligned}\end{equation}
	Furthermore, since $\hat{u}^\top W_\perp(z)=0$,
	\begin{equation}\label{eq:incremx}\delta_{ik}^\top W_{\perp}(z)  = m_{ik}e^{(k)\top}W_{\perp}(z) = m_{ik}W_{\perp,k}(z),
	\end{equation}
	where $W_{\perp,k}(z)$ denotes the $k$th element of the vector $W_{\perp}(z)$. A similar observation applies to a service completion at server $k$ for a class $i$ job. 
			
We first assume $\|W_{\perp}(z)\|\geq1$ and treat the case $\|W_{\perp}(z)\|<1$ at the end of the proof. With the test function (\ref{eq:ssctest}) and incremental change (\ref{eq:increment}), we write
			\begin{equation}\label{eq:ssc1}\begin{aligned}	f(z+e_{ik})-f(z) =& \|W_{\perp}(z)+\delta_{ik}\|^{n+1} - \|W_{\perp}(z)\|^{n+1} \\
					=&\frac{\|W_{\perp}(z)+\delta_{ik}\|^{2(n+1)} - \|W_{\perp}(z)\|^{2(n+1)}}{\|W_{\perp}(z)+\delta_{ik}\|^{n+1} + \|W_{\perp}(z)\|^{n+1}}.\end{aligned}\end{equation}
	Using (\ref{eq:incremx}) and the binomial expansion, the numerator can be written as
			$$\begin{aligned}
				&\|W_{\perp}(z)+\delta_{ik}\|^{2(n+1)}
				=\bigg[\|W_{\perp}(z)\|^2 + 2\delta_{ik}^\top W_{\perp}(z) + \|\delta_{ik}\|^2\bigg]^{n+1}
                \\
				=&\|W_{\perp}(z)\|^{2(n+1)} + (n+1)\bigg[2m_{ik}W_{\perp,k}(z) + \|\delta_{ik}\|^2\bigg]\|W_{\perp}(z)\|^{2n}  + A_1(z), 
			\end{aligned}$$
			where $A_1(z)\triangleq \sum_{\ell=2}^{n+1}{n+1\choose\ell}\big[2m_{ik}W_{\perp,k}(z) + \|\delta_{ik}\|^{2}\big]^{\ell}\|W_{\perp}(z)\|^{2(n+1-\ell)}.$ Substituting this into (\ref{eq:ssc1}), one has
			\begin{equation}\label{eq:ssc2}	f(z+e_{ik})-f(z) 
			=\frac{(n+1)\big[2m_{ik}W_{\perp,k}(z) + \|\delta_{ik}\|^2\big]\|W_{\perp}(z)\|^{2n}}{\|W_{\perp}(z)+\delta_{ik}\|^{n+1} + \|W_{\perp}(z)\|^{n+1}} + \frac{ A_1(z)}{\|W_{\perp}(z)+\delta_{ik}\|^{n+1} + \|W_{\perp}(z)\|^{n+1}}.\end{equation}
For the first term in (\ref{eq:ssc2}), dividing both the numerator and denominator by $\|W_{\perp}(z)\|^{n+1}$ gives
	$$\begin{aligned}				&\frac{(n+1)\big[2m_{ik}W_{\perp,k}(z) + \|\delta_{ik}\|^2\big]\|W_{\perp}(z)\|^{n-1}  }{\big[1 +\frac{ 2m_{ik}W_{\perp,k}(z) + \|\delta_{ik}\|^2}{\|W_{\perp}(z)\|^2}\big]^{\frac{n+1}{2}} + 1}\\
=&\frac{(n+1)\frac{\big[2m_{ik}W_{\perp,k}(z) + \|\delta_{ik}\|^2\big]}{\|W_{\perp}(z)\|^{2}}\|W_{\perp}(z)\|^{n+1}  }{\big[1 +\frac{ 2m_{ik}W_{\perp,k}(z) + \|\delta_{ik}\|^2}{\|W_{\perp}(z)\|^2}\big]^{\frac{n+1}{2}} + 1}\\
\leq &(n+1)\frac{\big[2m_{ik}W_{\perp,k}(z) + \|\delta_{ik}\|^2\big]}{2\|W_{\perp}(z)\|^{2}}\|W_{\perp}(z)\|^{n+1},
			\end{aligned}$$
where the second equality rearranges the expression into the form of Lemma~\ref{alge_inequality}, which yields the last inequality. 
For the second term, note that
			$$\begin{aligned}
				&\frac{ A_1(z)}{\|W_{\perp}(z)+\delta_{ik}\|^{n+1} + \|W_{\perp}(z)\|^{n+1}}
				\leq \frac{\sum_{\ell=2}^{n+1}{n+1\choose\ell}\big[2m_{ik}\big|W_{\perp,k}(z) \big|+ \|\delta_{ik}\|^{2}\big]^{\ell}\|W_{\perp}(z)\|^{2(n+1-\ell)}}{ \|W_{\perp}(z)\|^{n+1}}\\
				\leq& \sum_{\ell=2}^{n+1}{n+1\choose\ell}\big[2m_{ik}\big\|W_{\perp}(z) \big\|+ \|\delta_{ik}\|^{2}\big]^{\ell}\|W_{\perp}(z)\|^{n+1-2\ell} \triangleq
            \sum_{\ell=-(n+1)}^{n-1}\tilde c_{\ell}
            \|W_{\perp}(z)\|^{\ell}
            \leq
            \sum_{\ell=0}^{n-1}c'_{\ell}
            \|W_{\perp}(z)\|^{\ell},
			\end{aligned}$$
where the equality marked by $\triangleq$ follows by expanding the powers and collecting terms according to the powers of $\|W_{\perp}(z)\|$, with $\tilde{c}_{\ell}>0$ denoting the resulting constants. For the last inequality, since $\|W_{\perp}(z)\|\geq1$, the negative-power terms can be absorbed into the constant term, yielding some constants $c'_{\ell}>0$, $\ell=0,\ldots,n-1$. Therefore, $(\ref{eq:ssc2})$ becomes
$$\begin{aligned}	
f(z+e_{ik})-f(z) \leq &(n+1)\frac{\big[2m_{ik}W_{\perp,k}(z) + \|\delta_{ik}\|^2\big]}{2\|W_{\perp}(z)\|^{2}}\|W_{\perp}(z)\|^{n+1} + \sum_{\ell=0}^{n-1}c'_{\ell}\|W_{\perp}(z)\|^{\ell} \\
= &(n+1)m_{ik}W_{\perp,k}(z)\|W_{\perp}(z)\|^{n-1} +(n+1)\frac{ \|\delta_{ik}\|^2}{2}\|W_{\perp}(z)\|^{n-1} + \sum_{\ell=0}^{n-1}c'_{\ell}\|W_{\perp}(z)\|^{\ell} \\
\triangleq &(n+1)m_{ik}W_{\perp,k}(z)\|W_{\perp}(z)\|^{n-1}  + \sum_{\ell=0}^{n-1}c^{A}_{\ell}\|W_{\perp}(z)\|^{\ell}.
			\end{aligned}$$
	Here, we keep the first term explicitly and absorb all the remaining terms into the polynomial in $\|W_{\perp}(z)\|$, with the resulting constants denoted by $c^{A}_{\ell}>0$, $\ell=0,\ldots,n-1$. This establishes the first inequality for $\|W_{\perp}(z)\|\geq1$. The second inequality for $f(z-e_{ik})-f(z)$ follows similarly.
    
   Finally, consider $\|W_\perp(z)\|<1$. Since $\|\delta_{ik}\|\le m_{ik}$, the increments $f(z\pm e_{ik})-f(z)$ are uniformly bounded, and so are the leading terms $(n+1)m_{ik}W_{\perp,k}(z)\|W_\perp(z)\|^{n-1}$. Thus, by enlarging $c_0^A$ and $c_0^S$ if necessary, the same increment bounds hold for $\|W_\perp(z)\|<1$.
		\end{proof}

		\subsection{Proof of Lemma \ref{lem:ssc2}}\label{sec:ssc_lm2}
		\begin{proof}
For (\ref{eq:ub1}), first note from Lemma~\ref{lem:pd1} that
$$
\sum_{i=1}^{I}\lambda_i v_i\sum_{k\in\mathcal{K}(i)}
\mathbbm{1}\big(k=L^{(i)}(z)\big)
=\sum_{i=1}^{I}\lambda_i v_i=1,\text { and }\sum_{k=1}^{K}u_k=1.
$$
Moreover, by (\ref{eq:dik}),
$$
\begin{aligned}
&\sum_{i=1}^{I}\lambda_i\sum_{k\in\mathcal{K}(i)}
\frac{u_k}{\mu_{ik}}\mathbbm{1}\big(k=L^{(i)}(z)\big)
-\sum_{k=1}^{K}u_k\\
=&\sum_{i=1}^{I}\lambda_i\sum_{k\in\mathcal{K}(i)}
\left(v_i+\frac{d_{ik}}{\mu_{ik}}\right)
\mathbbm{1}\big(k=L^{(i)}(z)\big)-\sum_{k=1}^{K}u_k\\
=&\sum_{i=1}^{I}\lambda_i\sum_{k\in\mathcal{K}(i)}
\frac{d_{ik}}{\mu_{ik}}\mathbbm{1}\big(k=L^{(i)}(z)\big)\geq0.
\end{aligned}
$$
Using the definition of $W_{\perp}(z)$, we obtain 
\begin{equation}\label{eq:lm2}
\begin{aligned}
&\sum_{i=1}^{I}\lambda_i\sum_{k\in\mathcal{K}(i)}
m_{ik}W_{\perp,k}(z)\mathbbm{1}\big(k=L^{(i)}(z)\big)
-\sum_{k=1}^{K}W_{\perp,k}(z)\\
=&\sum_{i=1}^{I}\lambda_i\sum_{k\in\mathcal{K}(i)}
m_{ik}W_k(z)\mathbbm{1}\big(k=L^{(i)}(z)\big)
-\sum_{k=1}^{K}W_k(z)\\
&-\frac{1}{\|u\|^2}\left(\sum_{\ell=1}^{K}u_{\ell}W_{\ell}(z)\right)
\left(
\sum_{i=1}^{I}\lambda_i\sum_{k\in\mathcal{K}(i)}
\frac{u_k}{\mu_{ik}}\mathbbm{1}\big(k=L^{(i)}(z)\big)
-\sum_{k=1}^{K}u_k
\right)\\
\leq&
\sum_{i=1}^{I}\lambda_i\sum_{k\in\mathcal{K}(i)}
m_{ik}W_k(z)\mathbbm{1}\big(k=L^{(i)}(z)\big)
-\sum_{k=1}^{K}W_k(z),
\end{aligned}
\end{equation}
where the last inequality follows from the nonnegativity established above and $\sum_{\ell=1}^{K}u_{\ell}W_{\ell}(z)\geq0$.

We choose an optimal processing plan $x^*$ such that $x^*_{ik}>0$ for every basic activity $(i,k)$. Such a solution exists because, by definition, each basic activity has positive allocation in some optimal processing plan, and a convex combination of these optimal processing plans remains optimal and assigns strictly positive allocation to every basic activity. Define $\lambda_{ik}\triangleq x^*_{ik}\mu_{ik}$ for $(i,k)\in\mathcal{J}$. Then
$\lambda_i=\sum_{k\in\mathcal{K}(i)}\lambda_{ik}$ for each $i\in\mathcal{I}$ and $\sum_{i\in\mathcal{I}(k)}\lambda_{ik}m_{ik}=1$ for each $k\in\mathcal{K}$. Hence,
$$
\sum_{k=1}^{K}W_k(z)=\sum_{k=1}^{K}\sum_{i\in\mathcal{I}(k)}\lambda_{ik}m_{ik}W_k(z)=\sum_{i=1}^{I}\sum_{k\in\mathcal{K}(i)}\lambda_{ik}m_{ik}W_k(z).
$$
Moreover, by the definition of the WWTA policy,
$$
\sum_{k\in\mathcal{K}(i)}m_{ik}W_k(z)\mathbbm{1}\big(k=L^{(i)}(z)\big)=\min_{k\in\mathcal{K}(i)}m_{ik}W_k(z).
$$
Therefore, using $\lambda_i=\sum_{k\in\mathcal{K}(i)}\lambda_{ik}$, the last line of (\ref{eq:lm2}) can be written as
$$
\begin{aligned}
&\sum_{i=1}^{I}\lambda_i\min_{k'\in\mathcal{K}(i)}m_{ik'}W_{k'}(z)-\sum_{k=1}^{K}W_k(z)\\
=&\sum_{i=1}^{I}\sum_{k\in\mathcal{K}(i)}\lambda_{ik}\left[\min_{k'\in\mathcal{K}(i)}m_{ik'}W_{k'}(z)-m_{ik}W_k(z)\right].
\end{aligned}
$$
For activities with $d_{ik}>0$, complementary slackness gives $x^*_{ik}=0$ and hence $\lambda_{ik}=0$. We therefore obtain
$$\begin{aligned}&\sum_{i=1}^{I}\sum_{k\in \mathcal{K}(i)}\lambda_{ik}\left(\min_{k'\in \mathcal{K}(i)}m_{ik'}W_{k'} (z)- m_{ik}W_k(z)\right)\\
    =&\sum_{i=1}^{I}\sum_{k\in \mathcal{K}(i)}^{b}\lambda_{ik}v_i\left(\min_{k'\in \mathcal{K}(i)}\frac{1}{\mu_{ik'}v_i}W_{k'}(z) - \frac{1}{\mu_{ik}v_i}W_k(z)\right)\\
    =&\sum_{i=1}^{I}\sum_{k\in \mathcal{K}(i)}^{b}\lambda_{ik}v_i\left(\min_{k'\in \mathcal{K}(i)}\frac{1}{u_{k'}-d_{ik'}}W_{k'}(z) - \frac{1}{u_k}W_k(z)\right).\\
\end{aligned}$$
Note that the minimization over $k'\in\mathcal{K}(i)$ is over both basic and non-basic activities, as the WWTA policy does not distinguish between them when making routing decisions. Restricting the minimization to activities with $d_{ik}=0$ can only increase the minimum, and hence
$$
\min_{k'\in \mathcal{K}(i)}\frac{1}{u_{k'}-d_{ik'}}W_{k'}(z)\leq\min_{k''\in \mathcal{K}(i)}^{b}\frac{1}{u_{k''}}W_{k''}(z),
$$
where the superscript ``b" indicates restriction to activities with $d_{ik}=0$. Therefore,
\begin{equation}\label{eq:ssc_basic}
\begin{aligned}
&\sum_{i=1}^{I}\sum_{k\in \mathcal{K}(i)}^{b}\lambda_{ik}v_i\left(\min_{k'\in \mathcal{K}(i)}\frac{1}{u_{k'}-d_{ik'}}W_{k'}(z)-\frac{1}{u_k}W_k(z)\right)\\
\leq&\sum_{i=1}^{I}\sum_{k\in \mathcal{K}(i)}^{b}\lambda_{ik}v_i\left(\min_{k''\in \mathcal{K}(i)}^{b}\frac{1}{u_{k''}}W_{k''}(z)-\frac{1}{u_k}W_k(z)\right)\\
\leq&-\min_{(i,k):\,\lambda_{ik}>0}(\lambda_{ik}v_i)\sum_{i=1}^{I}\sum_{k\in \mathcal{K}(i):\lambda_{ik}>0}\left|\frac{1}{u_k}W_k(z)-\min_{k''\in \mathcal{K}(i)}^{b}\frac{1}{u_{k''}}W_{k''}(z)\right|.
\end{aligned}
\end{equation}

We next relate the bound in (\ref{eq:ssc_basic}) to $T_{(K)}(z)-T_{(1)}(z)$. By Assumption~\ref{assu2}, all servers communicate. Hence, there exists a path of basic activities connecting the servers attaining $T_{(1)}(z)$ and $T_{(K)}(z)$. We choose such a path with the fewest servers, so that the job classes along the path are distinct, and hence so are the activities along the path. Moreover, by the choice of $x^*$, all these activities satisfy $\lambda_{ik}>0$, so their terms appear in the last line of (\ref{eq:ssc_basic}). To illustrate the argument, consider an example path of the form
$$
(1)\stackrel{i_1}{\rightarrow} k_1\stackrel{i_2}{\rightarrow} k_2\stackrel{i_3}{\rightarrow} k_3\stackrel{i_4}{\rightarrow} (K),
$$
where each pair of consecutive servers is connected through the indicated job class by basic activities. Retaining only the terms corresponding to this path in the last line of (\ref{eq:ssc_basic}) and upper bounding all remaining terms by zero, we obtain
$$
\begin{aligned}	
    &-\sum_{i=1}^{I}\sum_{k\in \mathcal{K}(i):\lambda_{ik}>0}\left|\frac{1}{u_k}W_k(z)-\min_{k''\in \mathcal{K}(i)}^{b}\frac{1}{u_{k''}}W_{k''}(z)\right|\\
    \leq&-\left| T_{(1)}(z)-\min_{k''\in \mathcal{K}(i_1)}^{b}\frac{1}{u_{k''}}W_{k''}(z)\right|-\left| \frac{1}{u_{k_1}}W_{k_1}(z)-\min_{k''\in \mathcal{K}(i_1)}^{b}\frac{1}{u_{k''}}W_{k''}(z)\right|
    \\
    &-\left| \frac{1}{u_{k_1}}W_{k_1}(z)-\min_{k''\in \mathcal{K}(i_2)}^{b}\frac{1}{u_{k''}}W_{k''}(z)\right|
    -\left| \frac{1}{u_{k_2}}W_{k_2}(z)-\min_{k''\in \mathcal{K}(i_2)}^{b}\frac{1}{u_{k''}}W_{k''}(z)\right|\\
    &-\left| \frac{1}{u_{k_2}}W_{k_2}(z)-\min_{k''\in \mathcal{K}(i_3)}^{b}\frac{1}{u_{k''}}W_{k''}(z)\right|
    -\left| \frac{1}{u_{k_3}}W_{k_3}(z)-\min_{k''\in \mathcal{K}(i_3)}^{b}\frac{1}{u_{k''}}W_{k''}(z)\right|\\
    &-\left| \frac{1}{u_{k_3}}W_{k_3}(z)-\min_{k''\in \mathcal{K}(i_4)}^{b}\frac{1}{u_{k''}}W_{k''}(z)\right|
    -\left| T_{(K)}(z)-\min_{k''\in \mathcal{K}(i_4)}^{b}\frac{1}{u_{k''}}W_{k''}(z)\right|\\
    \leq&-\left|\frac{1}{u_{k_1}}W_{k_1}(z)-T_{(1)}(z)\right|-\left| \frac{1}{u_{k_1}}W_{k_1}(z)-\frac{1}{u_{k_2}}W_{k_2}(z)\right|\\
    &-\left| \frac{1}{u_{k_2}}W_{k_2}(z)-\frac{1}{u_{k_3}}W_{k_3}(z)\right|-\left| \frac{1}{u_{k_3}}W_{k_3}(z)-T_{(K)}(z)\right|\\
    \leq& -\left(T_{(K)}(z) - T_{(1)}(z)\right),
\end{aligned}$$
where the first inequality follows by retaining only the non-positive terms associated with the selected path. The second inequality follows by applying the triangle inequality to each pair of consecutive servers along the path, and the last inequality follows by applying the triangle inequality again along the entire path from $T_{(1)}(z)$ to $T_{(K)}(z)$. This proves (\ref{eq:ub1}).

For (\ref{eq:ub2}), using the definition of $W_{\perp,k}(z)$, we have
$$\begin{aligned}
    &-r\sum_{i=1}^{I}\lambda_i\sum_{k\in \mathcal{K}(i)}m_{ik}W_{\perp,k}(z)\mathbbm{1}\big(k=L^{(i)}({z})\big)\\
    =&r\sum_{i=1}^{I}\lambda_iv_i\sum_{k\in \mathcal{K}(i)}\frac{u_k}{u_k-d_{ik}}\left[\frac{1}{\|u\|^2}\sum_{\ell=1}^{K}u_{\ell}W_{\ell}(z)-\frac{1}{u_k}W_k(z) \right]\mathbbm{1}\big(k=L^{(i)}({z})\big)\\
    \leq&r\sum_{i=1}^{I}\lambda_iv_i\sum_{k\in \mathcal{K}(i)}\frac{u_k}{u_k-d_{ik}}\left[T_{(K)}(z)-T_{(1)}(z) \right]\mathbbm{1}\big(k=L^{(i)}({z})\big)\\
    \leq&r\max_{(i,k)}\frac{u_k}{u_k-d_{ik}}\sum_{i=1}^{I}\lambda_iv_i\sum_{k\in \mathcal{K}(i)}\left[T_{(K)}(z)-T_{(1)}(z) \right]\mathbbm{1}\big(k=L^{(i)}({z})\big)\\
    =& r\max_{(i,k)}\frac{u_k}{u_k-d_{ik}}\big[T_{(K)} (z)- T_{(1)}(z)\big].\end{aligned}$$
The first inequality follows since both $\frac{1}{\|u\|^2}\sum_{\ell=1}^{K}u_{\ell}W_{\ell}(z)$ and $\frac{1}{u_k}W_k(z)$ lie between $T_{(1)}(z)$ and $T_{(K)}(z)$. The second inequality follows by taking the maximum of $\frac{u_k}{u_k-d_{ik}}$, and the last equality follows from $\sum_{i=1}^{I}\lambda_iv_i=1$. This proves (\ref{eq:ub2}). 
\end{proof}

\section{Proofs in Section \ref{sec:outline2}}

\subsection{Proof of Lemma \ref{lem:nbcon}}\label{sec:nbcon}
\begin{proof}
    \begin{enumerate}
        \item
        To prove $(\ref{eq:nbcon_2})$, consider a non-basic activity $(i,k)$. We have
        \begin{align}
            &\E\left[d_{ik}\frac{1}{u_k}W_k({Z}^{(r)})\mathbbm{1}\left(k = L^{(i)}({Z}^{(r)})\right)\right] \label{eq:B1}\\
            \stackrel{(a)}{\leq}&\E\left[d_{ik}\frac{1}{u_k-d_{ik}}W_k({Z}^{(r)})\mathbbm{1}\left(k = L^{(i)}({Z}^{(r)})\right)\right]\label{eq:B2}\\		
            \stackrel{(b)}{\leq}& \E\left[d_{ik}\frac{1}{u_{k'}}W_{k'}({Z}^{(r)})\mathbbm{1}\left(k = L^{(i)}({Z}^{(r)})\right)\right], \label{eq:B3}
        \end{align}      
        where $(a)$ follows from $0<d_{ik}<u_k$. For $(b)$, since $(i,k)$ is a non-basic activity, there exists a server $k'$ such that activity $(i,k')$ is basic. On the event $k=L^{(i)}({Z}^{(r)})$, the definition of the WWTA policy gives $\frac{1}{\mu_{ik}}W_k({Z}^{(r)})\leq \frac{1}{\mu_{ik'}}W_{k'}({Z}^{(r)})$. Dividing both sides by $v_i$ gives
        \begin{align*}
            \frac{1}{u_k-d_{ik}}W_k({Z}^{(r)})\leq \frac{1}{u_{k'}-d_{ik'}}W_{k'}({Z}^{(r)}) \leq \frac{1}{u_{k'}}W_{k'}({Z}^{(r)}),\quad \text{ if } k=L^{(i)}({Z}^{(r)}).
        \end{align*}
        Therefore, $(\ref{eq:B2})-(\ref{eq:B1})\leq (\ref{eq:B3})-(\ref{eq:B1})$ implies
        $$\begin{aligned}
            0\leq&\E\left[d_{ik}\frac{d_{ik}}{(u_k-d_{ik})u_k}W_k({Z}^{(r)})\mathbbm{1}\left(k = L^{(i)}({Z}^{(r)})\right)\right]\\
            \leq & \E\left[d_{ik}\left(\frac{1}{u_{k'}}W_{k'}({Z}^{(r)})-\frac{1}{u_k}W_k({Z}^{(r)})\right)\mathbbm{1}\left(k = L^{(i)}({Z}^{(r)})\right)\right]\\
            {\leq}&\E\left[d_{ik}\left[\max_{k\in\mathcal{K}}\frac{1}{u_{k}}W_{k}({Z}^{(r)})-\min_{k\in\mathcal{K}}\frac{1}{u_{k}}W_{k}({Z}^{(r)})\right]\mathbbm{1}\left(k = L^{(i)}({Z}^{(r)})\right)\right]\\		
            \stackrel{(c)}{\leq}&d_{ik}\E\left[\left(T_{(K)}({Z}^{(r)})-T_{(1)}({Z}^{(r)})\right)^2\right]^{1/2}\Prob\left(k = L^{(i)}({Z}^{(r)})\right)^{1/2}\\
            \stackrel{(d)}{=}&O(r^{1/2}),
        \end{aligned}$$
        where $(c)$ follows from the Cauchy--Schwarz inequality, and $(d)$ follows from Proposition~\ref{prop:sscmoment} and (\ref{eq:nbcon_1}).
        Therefore, for a non-basic activity $(i,k)$, we have
        \begin{equation}\label{single_nb}
            \E\left[W_k({Z}^{(r)})\mathbbm{1}\left(k = L^{(i)}({Z}^{(r)})\right)\right]=O(r^{1/2}).   
        \end{equation}
        Furthermore, for a non-basic activity $(i,k)$,
        $$\begin{aligned}
            &\E\left[\bigg(\sum_{k'=1}^Ku_{k'}W_{k'}({Z}^{(r)})\bigg)\mathbbm{1}\left(k=L^{(i)}({Z}^{(r)})\right)\right]\\
            =&\E\left[\sum_{k'=1}^Ku^2_{k'}\bigg(\frac{1}{u_{k'}}W_{k'}({Z}^{(r)})-\frac{1}{u_{k}}W_{k}({Z}^{(r)})\bigg)\mathbbm{1}\left(k=L^{(i)}({Z}^{(r)})\right)\right]\\
            &+\sum_{k'=1}^Ku^2_{k'}\frac{1}{u_{k}}\E\left[W_{k}({Z}^{(r)})\mathbbm{1}\left(k=L^{(i)}({Z}^{(r)})\right)\right]\\
            \stackrel{(e)}{\leq}&\sum_{k'=1}^Ku^2_{k'}\E\left[\bigg(\max_{k\in\mathcal{K}}\frac{1}{u_{k}}W_{k}({Z}^{(r)})-\min_{k\in\mathcal{K}}\frac{1}{u_{k}}W_{k}({Z}^{(r)})\bigg)\mathbbm{1}\left(k=L^{(i)}({Z}^{(r)})\right)\right]+O(r^{1/2})\\
            \stackrel{(f)}{\leq}&\sum_{k'=1}^Ku^2_{k'}\E\bigg[\bigg(\max_{k\in\mathcal{K}}\frac{1}{u_{k}}W_{k}({Z}^{(r)})-\min_{k\in\mathcal{K}}\frac{1}{u_{k}}W_{k}({Z}^{(r)})\bigg)^2\bigg]^{1/2}\Prob\left(k=L^{(i)}({Z}^{(r)})\right)^{1/2}+O(r^{1/2})\\
            \stackrel{(g)}{=}&O(r^{1/2}),
        \end{aligned}$$
        where $(e)$ follows from (\ref{single_nb}), $(f)$ follows from the Cauchy--Schwarz inequality, and $(g)$ follows from Proposition~\ref{prop:sscmoment} and (\ref{eq:nbcon_1}). The proof of $(\ref{eq:nbcon_2})$ is complete. 
        \item We prove $(\ref{eq:nbcon_3})$ and $(\ref{eq:nbcon_4})$ together.
        Let $f({z}) = (\sum_{i=1}^I\sum_{k\in \mathcal{K}{(i)}}v_iz_{ik})^2$. Then the generator (\ref{eq:generator}) becomes
        $$\begin{aligned}
            Gf({z})&=2\sum_{i=1}^I\sum_{k\in \mathcal{K}{(i)}}\lambda_{i}^{(r)}v_i \bigg(\sum_{i'=1}^I\sum_{k'\in \mathcal{K}{(i')}}v_{i'}z_{i'k'}\bigg)\mathbbm{1}\left(k=L^{(i)}({z})\right)+\sum_{i=1}^I\sum_{k\in \mathcal{K}{(i)}}\lambda_{i}^{(r)}v^2_i \mathbbm{1}\left(k=L^{(i)}({z})
            \right)\\
            &-2\sum_{k=1}^{K}\sum_{i\in \mathcal{I}(k)}\mu_{ik}v_iP_{ik}({z})\bigg(\sum_{i'=1}^I\sum_{k'\in \mathcal{K}{(i')}}v_{i'}z_{i'k'}\bigg)+\sum_{k=1}^{K}\sum_{i\in \mathcal{I}(k)}\mu_{ik}v^2_iP_{ik}({z}).
        \end{aligned}$$
        It follows from Lemma \ref{lem:steady_exp} that
        \begin{equation}\label{equivleq1}
            \begin{aligned}
                &\sum_{k=1}^{K}\sum_{i\in \mathcal{I}(k)}\mu_{ik}v_i\E\bigg[P_{ik}({Z}^{(r)})\bigg(\sum_{i'=1}^I\sum_{k'\in \mathcal{K}{(i')}}v_{i'}Z^{(r)}_{i'k'}\bigg)\bigg]\\
                =&\sum_{i=1}^I\sum_{k\in \mathcal{K}{(i)}}\lambda_{i}^{(r)}v_i\E\bigg[ \bigg(\sum_{i'=1}^I\sum_{k'\in \mathcal{K}{(i')}}v_{i'}Z^{(r)}_{i'k'}\bigg)\mathbbm{1}\left(k=L^{(i)}({Z}^{(r)})\right)\bigg]+\frac{1}{2}\sum_{i=1}^I\sum_{k\in \mathcal{K}{(i)}}\lambda_{i}^{(r)}v^2_i \Prob\left(k=L^{(i)}({Z}^{(r)})\right)\\
                &+\frac{1}{2}\sum_{k=1}^{K}\sum_{i\in \mathcal{I}(k)}\mu_{ik}v^2_i\E\left[P_{ik}({Z}^{(r)})
                \right].\end{aligned}
        \end{equation}
        Let $f({z}) = \left(\sum_{i=1}^I\sum_{k\in \mathcal{K}{(i)}}^{b}v_iz_{ik}\right)^2$, i.e., we consider only activities with $d_{ik}=0$. 
        Applying a similar generator calculation and Lemma~\ref{lem:steady_exp}, we obtain
        \begin{equation}\label{equivleq2}
            \begin{aligned}
                &\sum_{k=1}^{K}\sum_{i\in \mathcal{I}(k)}^{b}\mu_{ik}v_i\E\bigg[P_{ik}({Z}^{(r)})\bigg(\sum_{i'=1}^I\sum_{k'\in \mathcal{K}{(i')}}^{b}v_{i'}Z^{(r)}_{i'k'}\bigg)\bigg]\\
                =&\sum_{i=1}^I\sum_{k\in \mathcal{K}{(i)}}^{b}\lambda_{i}^{(r)}v_i\E\bigg[ \bigg(\sum_{i'=1}^I\sum_{k'\in \mathcal{K}{(i')}}^{b}v_{i'}Z^{(r)}_{i'k'}\bigg)\mathbbm{1}\left(k=L^{(i)}({Z}^{(r)})\right)\bigg]+\frac{1}{2}\sum_{i=1}^I\sum_{k\in \mathcal{K}{(i)}}^{b}\lambda_{i}^{(r)}v^2_i \Prob\left(k=L^{(i)}({Z}^{(r)})\right)\\
                &+\frac{1}{2}\sum_{k=1}^{K}\sum_{i\in \mathcal{I}(k)}^{b}\mu_{ik}v^2_i\E\left[P_{ik}({Z}^{(r)})
                \right].\end{aligned}
        \end{equation}
        Subtracting (\ref{equivleq2}) from (\ref{equivleq1}), we have
        \begin{equation}\label{equivleq3}\begin{aligned}
                &\sum_{k=1}^{K}\sum_{i\in \mathcal{I}(k)}^{nb}\mu_{ik}v_i\E\bigg[P_{ik}({Z}^{(r)})\bigg(\sum_{i'=1}^I\sum_{k'\in \mathcal{K}{(i')}}v_{i'}Z^{(r)}_{i'k'}\bigg)\bigg]\\
                &+ \sum_{k=1}^{K}\sum_{i\in \mathcal{I}(k)}^{b}\mu_{ik}v_i\E\bigg[P_{ik}({Z}^{(r)})\bigg(\sum_{i'=1}^I\sum_{k'\in \mathcal{K}{(i')}}^{nb}v_{i'}Z^{(r)}_{i'k'}\bigg)\bigg]\\
                =&\sum_{i=1}^I\sum_{k\in \mathcal{K}{(i)}}^{nb}\lambda_{i}^{(r)}v_i\E\bigg[ \bigg(\sum_{i'=1}^I\sum_{k'\in \mathcal{K}{(i')}}v_{i'}Z^{(r)}_{i'k'}\bigg)\mathbbm{1}\left(k=L^{(i)}({Z}^{(r)})\right)\bigg]\\
                &+\sum_{i=1}^I\sum_{k\in \mathcal{K}{(i)}}^{b}\lambda_{i}^{(r)}v_i\E\bigg[ \bigg(\sum_{i'=1}^I\sum_{k'\in \mathcal{K}{(i')}}^{nb}v_{i'}Z^{(r)}_{i'k'}\bigg)\mathbbm{1}\left(k=L^{(i)}({Z}^{(r)})\right)\bigg]\\
                &+\frac{1}{2}\sum_{i=1}^I\sum_{k\in \mathcal{K}{(i)}}^{nb}\lambda_{i}^{(r)}v^2_i \Prob\left(k=L^{(i)}({Z}^{(r)})\right)+\frac{1}{2}\sum_{k=1}^{K}\sum_{i\in \mathcal{I}(k)}^{nb}\mu_{ik}v^2_i\E\left[P_{ik}({Z}^{(r)})\right]\\  
                = &\sum_{i=1}^I\sum_{k\in \mathcal{K}{(i)}}^{nb}\lambda_{i}^{(r)}v_i\E\bigg[ \bigg(\sum_{i'=1}^I\sum_{k'\in \mathcal{K}{(i')}}v_{i'}Z^{(r)}_{i'k'}\bigg)\mathbbm{1}\left(k=L^{(i)}({Z}^{(r)})\right)\bigg]\\
                &+\sum_{i=1}^I\sum_{k\in \mathcal{K}{(i)}}^{b}\lambda_{i}^{(r)}v_i\E\bigg[ \bigg(\sum_{i'=1}^I\sum_{k'\in \mathcal{K}{(i')}}^{nb}v_{i'}Z^{(r)}_{i'k'}\bigg)\mathbbm{1}\left(k=L^{(i)}({Z}^{(r)})\right)\bigg]+O(r),\\
        \end{aligned}\end{equation}
        where the last equality follows from (\ref{eq:nbcon_0}) and (\ref{eq:nbcon_1}). 
        
        Let $f({z})=(\sum_{i=1}^{I}\sum_{k\in \mathcal{K}(i)}^{nb}v_iz_{ik})^2$. 
        Applying the same calculation, together with Lemma~\ref{lem:steady_exp}, Lemma~\ref{lem:idle}, and (\ref{eq:nbcon_1}), we obtain
        \begin{equation}\label{equivleq4}
            \begin{aligned}&\sum_{i=1}^{I}\sum_{k\in \mathcal{K}(i)}^{nb}\mu_{ik}v_i\E\bigg[P_{ik}({Z}^{(r)})\bigg(\sum_{i'=1}^{I}\sum_{k'\in \mathcal{K}(i')}^{nb}v_{i'}Z_{i'k'}^{(r)}\bigg)\bigg]\\
                &-\sum_{i=1}^{I}\sum_{k\in \mathcal{K}(i)}^{nb}\lambda^{(r)}_iv_i\E\bigg[\bigg(\sum_{i'=1}^{I}\sum_{k'\in \mathcal{K}(i')}^{nb}v_{i'}Z_{i'k'}^{(r)}\bigg)\mathbbm{1}\left(k=L^{(i)}({Z}^{(r)})\right)\bigg]\\
                =&\frac{1}{2}\sum_{i=1}^{I}\sum_{k\in \mathcal{K}(i)}^{nb}\lambda^{(r)}_iv^2_i\Prob\left(k=L^{(i)}({Z}^{(r)})\right)+\frac{1}{2}\sum_{k=1}^{K}\sum_{i\in \mathcal{I}(k)}^{nb}\mu_{ik}v^2_i\E\left[P_{ik}({Z}^{(r)})\right]\\
                =&O(r).\end{aligned}\end{equation}
        
        Adding (\ref{equivleq3}) to (\ref{equivleq4}) and rearranging terms yields
        $$\begin{aligned}
            &\sum_{i=1}^I\sum_{k\in \mathcal{K}{(i)}}^{nb}\lambda_{i}^{(r)}v_i\E\bigg[ \bigg(\sum_{i'=1}^I\sum_{k'\in \mathcal{K}{(i')}}v_{i'}Z^{(r)}_{i'k'}\bigg)\mathbbm{1}\left(k=L^{(i)}({Z}^{(r)})\right)\bigg]\\
            &-\sum_{k=1}^{K}\sum_{i\in \mathcal{I}(k)}^{nb}\mu_{ik}v_i\E\bigg[P_{ik}({Z}^{(r)})\bigg(\sum_{i'=1}^I\sum_{k'\in \mathcal{K}{(i')}}v_{i'}Z^{(r)}_{i'k'}\bigg)\bigg] + O(r) \\
            =& r\E\bigg[\bigg(\sum_{i=1}^I\sum_{k\in \mathcal{K}{(i)}}^{nb}v_iZ^{(r)}_{ik}\bigg)\bigg]-\sum_{k=1}^{K}\sum_{i\in \mathcal{I}(k)}d_{ik}\E\bigg[P_{ik}({Z}^{(r)})\bigg(\sum_{i'=1}^I\sum_{k'\in \mathcal{K}{(i')}}^{nb}v_{i'}Z^{(r)}_{i'k'}\bigg)\bigg]\\
            &-\sum_{k=1}^{K}u_k\E\bigg[\bigg(\sum_{i'=1}^I\sum_{k'\in \mathcal{K}{(i')}}^{nb}v_{i'}Z^{(r)}_{i'k'}\bigg)\mathbbm{1}\left(\sum_{i\in \mathcal{I}(k)}Z^{(r)}_{ik} = 0\right)\bigg].
        \end{aligned}$$
        Note that $d_{ik}>0$ only for activities included in the $nb$ sums. Using (\ref{eq:relation}), we can rewrite the equation above as
        \begin{equation}\label{two_relation}
            \begin{aligned}
                &r\E\bigg[\bigg(\sum_{i=1}^I\sum_{k\in \mathcal{K}{(i)}}^{nb}v_iZ^{(r)}_{ik}\bigg)\bigg] +\sum_{k=1}^{K}\sum_{i\in \mathcal{I}(k)}^{nb}\mu_{ik}v_i\E\bigg[P_{ik}({Z}^{(r)})\bigg(\sum_{k'=1}^{K}u_{k'}W_{k'}(Z^{(r)})\bigg)\bigg]\\
                =&\sum_{i=1}^I\sum_{k\in \mathcal{K}{(i)}}^{nb}\lambda_{i}^{(r)}v_i\E\bigg[ \bigg(\sum_{i'=1}^I\sum_{k'\in \mathcal{K}{(i')}}v_{i'}Z^{(r)}_{i'k'}\bigg)\mathbbm{1}\left(k=L^{(i)}({Z}^{(r)})\right)\bigg]\\
                &+\sum_{k=1}^{K}\sum_{i\in \mathcal{I}(k)}^{nb}\mu_{ik}v_i\E\bigg[P_{ik}({Z}^{(r)})\bigg(\sum_{i'=1}^I\sum_{k'\in \mathcal{K}{(i')}}^{nb}\frac{d_{i'k'}}{\mu_{i'k'}}Z^{(r)}_{i'k'}\bigg)\bigg] \\
                &+\sum_{k=1}^{K}\sum_{i\in \mathcal{I}(k)}^{nb}d_{ik}\E\bigg[P_{ik}({Z}^{(r)})\bigg(\sum_{i'=1}^I\sum_{k'\in \mathcal{K}{(i')}}^{nb}v_{i'}Z^{(r)}_{i'k'}\bigg)\bigg]\\
                &+\sum_{k=1}^{K}u_k\E\bigg[\bigg(\sum_{i'=1}^I\sum_{k'\in \mathcal{K}{(i')}}^{nb}v_{i'}Z^{(r)}_{i'k'}\bigg)\mathbbm{1}\bigg(\sum_{i\in \mathcal{I}(k)}Z^{(r)}_{ik} = 0\bigg)\bigg] + O(r).
            \end{aligned}
        \end{equation}
        We next bound each term on the right-hand side of (\ref{two_relation}). By (\ref{eq:relation}) and (\ref{eq:nbcon_2}), the first term satisfies
        $$\begin{aligned}
            &\sum_{i=1}^I\sum_{k\in \mathcal{K}{(i)}}^{nb}\lambda_{i}^{(r)}v_i\E\bigg[ \bigg(\sum_{i'=1}^I\sum_{k'\in \mathcal{K}{(i')}}v_{i'}Z^{(r)}_{i'k'}\bigg)\mathbbm{1}\left(k=L^{(i)}({Z}^{(r)})\right)\bigg]\\
            \leq&\sum_{i=1}^I\sum_{k\in \mathcal{K}{(i)}}^{nb}\lambda_{i}^{(r)}v_i\E\bigg[ \bigg(\sum_{k'=1}^{K}u_{k'}W_{k'}(Z^{(r)})\bigg)\mathbbm{1}\left(k=L^{(i)}({Z}^{(r)})\right)\bigg]\\
            =&O(r^{1/2}).
        \end{aligned}$$
        For the second term, we take $c_1 >0$ such that $\frac{d_{ik}}{\mu_{ik}}\leq c_1v_i$, for all $(i,k) \in\mathcal{J}$. By (\ref{equivleq4}) and $(\ref{eq:relation})$, 
        $$\begin{aligned}
            &\sum_{k=1}^{K}\sum_{i\in \mathcal{I}(k)}^{nb}\mu_{ik}v_i\E\bigg[P_{ik}({Z}^{(r)})\bigg(\sum_{i'=1}^I\sum_{k'\in \mathcal{K}{(i')}}^{nb}\frac{d_{i'k'}}{\mu_{i'k'}}Z^{(r)}_{i'k'}\bigg)\bigg]\\
            \leq &c_1\sum_{i=1}^{I}\sum_{k\in \mathcal{K}(i)}^{nb}\mu_{ik}v_i\E\bigg[P_{ik}({Z}^{(r)})\bigg(\sum_{i'=1}^{I}\sum_{k'\in \mathcal{K}(i')}^{nb}v_{i'}Z_{i'k'}^{(r)}\bigg)\bigg]\\
            \stackrel{(\ref{equivleq4})}{=} &c_1\sum_{i=1}^{I}\sum_{k\in \mathcal{K}(i)}^{nb}\lambda^{(r)}_iv_i\E\bigg[\left(\sum_{i'=1}^{I}\sum_{k'\in \mathcal{K}(i')}^{nb}v_{i'}Z_{i'k'}^{(r)}\right)\mathbbm{1}\left(k=L^{(i)}({Z}^{(r)})\right)\bigg] +O(r)\\
            \stackrel{(\ref{eq:relation})}{\leq} &c_1\sum_{i=1}^I\sum_{k\in \mathcal{K}{(i)}}^{nb}\lambda_{i}v_i\E\bigg[ \bigg(\sum_{k'=1}^{K}u_{k'}W_{k'}(Z^{(r)})\bigg)\mathbbm{1}\left(k=L^{(i)}({Z}^{(r)})\right)\bigg]+O(r)\\
            =&O(r^{1/2}).
        \end{aligned}
        $$
        For the third term, choose $c_2>0$ such that $d_{ik}\leq c_2\mu_{ik}v_i$ for all $(i,k) \in\mathcal{J}$. By (\ref{equivleq4}), (\ref{eq:relation}), and (\ref{eq:nbcon_2}), we obtain 
        $$
\sum_{k=1}^{K}\sum_{i\in \mathcal{I}(k)}^{nb}d_{ik}\E\bigg[P_{ik}({Z}^{(r)})\bigg(\sum_{i'=1}^I\sum_{k'\in \mathcal{K}{(i')}}^{nb}v_{i'}Z^{(r)}_{i'k'}\bigg)\bigg]
=O(r^{1/2}).
$$
        For the fourth term, by Lemma \ref{lem:neg}, we have
$$\begin{aligned}&\sum_{k=1}^{K}u_k\E\left[\bigg(\sum_{i'=1}^I\sum_{k'\in \mathcal{K}{(i')}}^{nb}v_{i'}Z^{(r)}_{i'k'}\bigg)\mathbbm{1}\bigg(\sum_{i\in \mathcal{I}(k)}Z^{(r)}_{ik} = 0\bigg)\right]\\
            \leq& \sum_{k=1}^{K}u_k\E\left[\bigg(\sum_{k'=1}^{K}u_{k'}W_{k'}(Z^{(r)})\bigg)\mathbbm{1}\bigg(\sum_{i\in \mathcal{I}(k)}Z^{(r)}_{ik} = 0\bigg)\right]\\
            =&O(r^{1/2}).
        \end{aligned}$$
        Putting these bounds together, (\ref{two_relation}) therefore implies
        $$
        r\E\bigg[\bigg(\sum_{i=1}^I\sum_{k\in \mathcal{K}{(i)}}^{nb}v_iZ^{(r)}_{ik}\bigg)\bigg]+\sum_{k=1}^{K}\sum_{i\in \mathcal{I}(k)}^{nb}\mu_{ik}v_i\E\left[P_{ik}({Z}^{(r)})\bigg(\sum_{k'=1}^{K}u_{k'}W_{k'}(Z^{(r)})\bigg)\right]= O(r^{1/2}).$$
       Both terms on the left-hand side are nonnegative. Since $\mu_{ik}v_i>0$ for every activity $(i,k)$, each individual term in the second sum is therefore $O(r^{1/2})$. Hence, (\ref{eq:nbcon_3}) and (\ref{eq:nbcon_4}) follow.
    \end{enumerate}
\end{proof}

		\subsection{Proof of Lemma \ref{lem:mbound1}}\label{sec:mbound1}
		\begin{proof}
			By rearranging $(\ref{equivleq1})$, we have
			$$
			\begin{aligned}
				&\frac{1}{2}\sum_{i=1}^I\sum_{k\in \mathcal{K}{(i)}}\lambda_{i}^{(r)}v^2_i \Prob\left(k=L^{(i)}({Z}^{(r)})\right)+\frac{1}{2}\sum_{k=1}^{K}\sum_{i\in \mathcal{I}(k)}\mu_{ik}v^2_i\E\left[P_{ik}({Z}^{(r)})
				\right]\\
				=&\sum_{k=1}^{K}\sum_{i\in \mathcal{I}(k)}\mu_{ik}v_i\E\bigg[P_{ik}({Z}^{(r)})\bigg(\sum_{i'=1}^I\sum_{k'\in \mathcal{K}{(i')}}v_{i'}Z^{(r)}_{i'k'}\bigg)\bigg]\\
				&-\sum_{i=1}^I\sum_{k\in \mathcal{K}{(i)}}\lambda_{i}^{(r)}v_i\E\bigg[ \bigg(\sum_{i'=1}^I\sum_{k'\in \mathcal{K}{(i')}}v_{i'}Z^{(r)}_{i'k'}\bigg)\mathbbm{1}\left(k=L^{(i)}({Z}^{(r)})\right)\bigg]\\
				=&r\E\bigg[\bigg(\sum_{i=1}^I\sum_{k\in \mathcal{K}{(i)}}v_iZ^{(r)}_{ik}\bigg)\bigg] -\sum_{k=1}^{K}\sum_{i\in \mathcal{I}(k)}d_{ik}\E\bigg[P_{ik}({Z}^{(r)})\bigg(\sum_{i'=1}^I\sum_{k'\in \mathcal{K}{(i')}}v_{i'}Z^{(r)}_{i'k'}\bigg)\bigg]\\
				&-\sum_{k=1}^{K}u_k\E\bigg[\bigg(\sum_{i'=1}^I\sum_{k'\in \mathcal{K}{(i')}}v_{i'}Z^{(r)}_{i'k'}\bigg)\mathbbm{1}\bigg(\sum_{i\in \mathcal{I}(k)}Z^{(r)}_{ik} = 0\bigg)\bigg].\\
			\end{aligned}			
			$$
			Rearranging the terms above gives
			$$\begin{aligned}
				&r\E\bigg[\bigg(\sum_{i=1}^I\sum_{k\in \mathcal{K}{(i)}}v_iZ^{(r)}_{ik}\bigg)\bigg] \\
				=&\sum_{k=1}^{K}\sum_{i\in \mathcal{I}(k)}d_{ik}\E\bigg[P_{ik}({Z}^{(r)})\bigg(\sum_{i'=1}^I\sum_{k'\in \mathcal{K}{(i')}}v_{i'}Z^{(r)}_{i'k'}\bigg)\bigg] + \sum_{k=1}^{K}u_k\E\bigg[\bigg(\sum_{i'=1}^I\sum_{k'\in \mathcal{K}{(i')}}v_{i'}Z^{(r)}_{i'k'}\bigg)\mathbbm{1}\bigg(\sum_{i\in \mathcal{I}(k)}Z^{(r)}_{ik} = 0\bigg)\bigg]\\
				&+\frac{1}{2}\sum_{i=1}^I\sum_{k\in \mathcal{K}{(i)}}\lambda_{i}^{(r)}v^2_i \Prob\left(k=L^{(i)}({Z}^{(r)})\right)+\frac{1}{2}\sum_{k=1}^{K}\sum_{i\in \mathcal{I}(k)}\mu_{ik}v^2_i\E\left[P_{ik}({Z}^{(r)})
				\right]\\
				\stackrel{(\ref{eq:relation})}{\leq}&\sum_{k=1}^{K}\sum_{i\in \mathcal{I}(k)}d_{ik}\E\bigg[P_{ik}({Z}^{(r)})\bigg(\sum_{k'=1}^Ku_{k'}W_{k'}\big({Z}^{(r)}\big)\bigg)\bigg]+ \sum_{k=1}^{K}u_k\E\bigg[\bigg(\sum_{k'=1}^Ku_{k'}W_{k'}\big({Z}^{(r)}\big)\bigg)\mathbbm{1}\bigg(\sum_{i\in \mathcal{I}(k)}Z^{(r)}_{ik} = 0\bigg)\bigg]\\
				&+\frac{1}{2}\sum_{i=1}^I\lambda_{i}v^2_i +\frac{1}{2}\sum_{k=1}^{K}\sum_{i\in \mathcal{I}(k)}\mu_{ik}v^2_i\\ 
				\leq &O(r^{1/2})+M_1,
			\end{aligned}$$
			where the last inequality follows from (\ref{eq:nbcon_3}) and Lemma~\ref{lem:neg}, and $M_1 \triangleq  \frac{1}{2}\sum_{i=1}^I\lambda_{i}v^2_i +\frac{1}{2}\sum_{k=1}^{K}\sum_{i\in \mathcal{I}(k)}\mu_{ik}v^2_i$ does not depend on $r$. Hence, $M_1+O(r^{1/2})$ is bounded by some constant $M>0$ for $r\in(0,r_0)$, which proves Lemma~\ref{lem:mbound1}.
		\end{proof}
		
    \section{Proofs in Section \ref{sec:mainproof}}
    \subsection{Proof of Corollary \ref{coro:ld_workload}}\label{sec:ld_workload}
    \begin{proof} 
        Denote $$\tilde{X}^{(r)}\triangleq r \left(\sum_{i=1}^I\sum_{k\in \mathcal{K}(i)}v_iZ^{(r)}_{ik}\right).$$
        Then by Proposition \ref{prop:ld}, 
        \begin{equation}\label{eq:neg1}
            \tilde{X}^{(r)}\stackrel{d}{\rightarrow} \tilde{X},\quad \text{as }r \downarrow 0.
        \end{equation}
        Denote $$U^{(r)}\triangleq r\left( \sum_{k=1}^Ku_kW_k({Z}^{(r)})\right).$$
        Then by Lemma \ref{lem:nbcon} and $(\ref{eq:relation})$, the difference
        $$\begin{aligned}
            \E\left[U^{(r)}-\tilde{X}^{(r)}\right] = r\E\left( \sum_{i=1}^I\sum_{k\in \mathcal{K}(i)}\frac{d_{ik}}{\mu_{ik}}Z^{(r)}_{ik}\right) \rightarrow 0,\quad  \text{as }r \downarrow 0.
        \end{aligned}$$
        Since $U^{(r)}-\tilde{X}^{(r)}\geq 0$, this implies 
        $$\left|U^{(r)}-\tilde{X}^{(r)}\right|\stackrel{L^1}{\rightarrow} 0,\quad  \text{as }r \downarrow 0,$$
       and hence
        \begin{equation}{\label{eq:neg2}}\begin{aligned}
                &\left|U^{(r)}-\tilde{X}^{(r)}\right|\stackrel{p}{\rightarrow} 0, \quad \text{as }r \downarrow 0.
            \end{aligned}
        \end{equation}
        Combining (\ref{eq:neg1}) with (\ref{eq:neg2}), we have
        $$U^{(r)}\stackrel{d}{\rightarrow }\tilde{X}, \quad \text{as }r \downarrow 0.$$
    \end{proof}

\subsection{Proof of Lemma \ref{lem:interchg}}\label{sec:interchg}

\begin{proof}
    For each $i\in\mathcal{I}$, we show that
    $$
    \begin{aligned}
        & \lim_{r\downarrow 0}\bigg(\lambda_{i}v_{i}\phi^{(r)}(\theta)
        -\sum_{k\in \mathcal{K}(i)}u_k\phi^{(r)}_{ik}(\theta)\bigg)\\
        =& \lim_{r\downarrow 0}\E\bigg[\lim_{t\downarrow 0}\bigg(\lambda_{i}v_i-\sum_{k\in \mathcal{K}(i)}u_kP_{ik}({Z}^{(r)})\bigg)e^{r\theta\left(t\sum_{k\in \mathcal{K}(i)}v_{i}Z^{(r)}_{ik}\right)}f_{\theta}(Z^{(r)})\bigg]\\
        \stackrel{(i)}{=}& \lim_{r\downarrow 0}\lim_{t\downarrow 0}\E\bigg[\bigg(\lambda_{i}v_i-\sum_{k\in \mathcal{K}(i)}u_kP_{ik}({Z}^{(r)})\bigg)e^{r\theta\left(t\sum_{k\in \mathcal{K}(i)}v_{i}Z^{(r)}_{ik}\right)}f_{\theta}(Z^{(r)})\bigg]\\
        \stackrel{(ii)}{=}& \lim_{t\downarrow 0}\lim_{r\downarrow 0}\E\bigg[\bigg(\lambda_{i}v_i-\sum_{k\in \mathcal{K}(i)}u_kP_{ik}({Z}^{(r)})\bigg)e^{r\theta\left(t\sum_{k\in \mathcal{K}(i)}v_{i}Z^{(r)}_{ik}\right)}f_{\theta}(Z^{(r)})\bigg]\\
        \stackrel{(iii)}{=}&\lim_{t\downarrow 0}0=0,
    \end{aligned}
    $$
    where $(i)$ follows from the bounded convergence theorem, since the integrand is bounded for $\theta\leq0$, and $(iii)$ follows directly from $(\ref{eq:helper})$. To justify $(ii)$, we state the Moore--Osgood theorem in \cite{Lawr1946} as follows: 
    \begin{theorem}[Moore--Osgood]{\label{thm:moore}}
        If $ \lim _{x \rightarrow p} f(x, y)$ exists pointwise for each $y$ different from  $q$  and if  $\lim _{y \rightarrow q} f(x, y)$ converges uniformly for $x \neq p$,  then the double limit and the iterated limits exist and are equal, i.e.,
        $$\lim _{(x, y) \rightarrow(p, q)} f(x, y)=\lim _{x \rightarrow p} \lim _{y \rightarrow q} f(x, y)=\lim _{y \rightarrow q} \lim _{x \rightarrow p} f(x, y).$$
    \end{theorem}
    To apply the theorem, for each $i\in\mathcal I$, define
    $$
    \begin{aligned}g_i(r,t)\triangleq
        &\E\bigg[\bigg(\lambda_{i}v_i-\sum_{k\in \mathcal{K}(i)}u_kP_{ik}({Z}^{(r)})\bigg)e^{r\theta\left(t\sum_{k\in \mathcal{K}(i)}v_{i}Z^{(r)}_{ik}\right)}f_{\theta}(Z^{(r)})\bigg].\end{aligned}$$
   We apply the theorem on the domain $r\in(0,r_0)$ and $t>0$ and verify its two conditions. First, by $(\ref{eq:helper})$,
    $\lim_{r\downarrow 0}g_i(r ,t)$ exists pointwise for each $t> 0$. To verify the second condition, we use the following lemma, proved in Appendix~\ref{sec:moorecon1}.
    \begin{lemma}{\label{lem:moorecon1}}
        For each $i\in\mathcal{I}$, $\lim_{t\downarrow 0}g_i(r,t)$ converges uniformly for $r\in(0,r_0)$.
    \end{lemma}
The conditions of the Moore--Osgood theorem are therefore satisfied, so the limits can be interchanged, completing the proof of $(ii)$. 
    
\end{proof}
			
\subsection{Proof of Lemma \ref{lem:moorecon1}}\label{sec:moorecon1}

\begin{proof}
   To prove uniform convergence, it suffices to bound, for $i\in\mathcal{I}$,
    $$\begin{aligned}
        & |g_i(r ,t)-g_i(r ,0)|\\
        =&\left|\E\left[\bigg(\lambda_{i}v_i -\sum_{k\in \mathcal{K}(i)}u_kP_{ik}({Z}^{(r)})\bigg)e^{r\theta\left(\sum_{i=1}^I\sum_{k\in \mathcal{K}(i)}v_iZ^{(r)}_{ik}\right)}\left(e^{r\theta t\sum_{k\in \mathcal{K}(i)}v_iZ^{(r)}_{ik}}-1\right)\right]\right|\\
        \leq & \E\left[\left|\lambda_{i}v_i -\sum_{k\in \mathcal{K}(i)}u_kP_{ik}({Z}^{(r)})\right|\left|e^{r\theta\left(\sum_{i=1}^I\sum_{k\in \mathcal{K}(i)}v_iZ^{(r)}_{ik}\right)}\right|\left|e^{r\theta t\sum_{k\in \mathcal{K}(i)}v_{i}Z^{(r)}_{ik}}-1\right|\right]\\
        \leq & 2\E\left|\left(e^{r\theta t\sum_{k\in \mathcal{K}(i)}v_{i}Z^{(r)}_{ik}}-1\right)\right|\\
        \stackrel{(a)}{\leq}  & 2|\theta|tr \sum_{k\in \mathcal{K}(i)}v_i \E\left( Z^{(r)}_{ik}\right)\\
        \stackrel{(b)}{\leq} & 2|\theta|v_itM_0,\\
    \end{aligned}
    $$
    where $(a)$ follows from the inequality $1-e^{-x}\leq x$ for $x\geq 0$, and $(b)$ follows from Lemma~\ref{lem:mbound1}, which ensures that there exists $M_0>0$ such that $r \sum_{k\in \mathcal{K}(i)}\E\left(Z^{(r)}_{ik}\right)\leq M_0$ for all $r\in (0,r_0)$. If $\theta=0$, the result is immediate. For $\theta<0$ and any $\xi>0$, let
$$\delta_i=\frac{\xi}{3v_iM_0|\theta|}.$$ Then, for all $r\in(0,r_0)$ and $0<t<\delta_i$,
    $$|g_i(r ,t)-g_i(r ,0)|\leq 2v_iM_0 |\theta| \frac{\xi}{3v_iM_0|\theta|}< \xi.$$
    Therefore, for each $i\in\mathcal{I}$, $\lim_{t\downarrow 0}g_i(r,t)$ converges uniformly for all $r\in(0,r_0)$.
\end{proof}		
	\addtocontents{toc}{\endgroup}
		\end{appendix}

	\end{document}